\documentclass[11pt]{amsart}
\usepackage{amsmath,amssymb,latexsym,soul}
\usepackage{cite,mathrsfs,amsfonts}
\usepackage{color,enumitem,graphicx}
\usepackage[colorlinks=true,urlcolor=blue,
citecolor=red,linkcolor=blue,linktocpage,pdfpagelabels,
bookmarksnumbered,bookmarksopen]{hyperref}
\usepackage[english]{babel}

\usepackage[xetex,a4paper,left=2.6cm,right=2.6cm,top=2.74cm,bottom=2.74cm]{geometry}

\usepackage[hyperpageref]{backref}

\usepackage[T1]{fontenc}
\usepackage{lmodern}

\numberwithin{equation}{section}

\newtheorem{theorem}{Theorem}[section]
\newtheorem{lemma}[theorem]{Lemma}

\newtheorem{proposition}[theorem]{Proposition}
\newtheorem{corollary}[theorem]{Corollary}
\theoremstyle{definition}
\newtheorem{definition}[theorem]{Definition}
\theoremstyle{remark}
\newtheorem{remark}[theorem]{Remark}

\newcommand{\J}{{\mathbb J}}
\newcommand{\I}{{\mathbb I}}

\newcommand{\C}{{\mathbb C}}
\newcommand{\A}{{\mathbb A}}
\newcommand{\M}{{\mathbb M}}
\newcommand{\R}{{\mathbb R}}
\newcommand{\eps}{\varepsilon}
\renewcommand{\Re}{\mathfrak{Re}}
\renewcommand{\Im}{\mathfrak{Im}}

\title[Soliton dynamics for fractional NLS]{Soliton dynamics for
  fractional Schr\"odinger equations}

\author[S.\ Secchi]{Simone Secchi}
\thanks{S.\ Secchi is supported by 2009 PRIN \emph{Critical Point Theory and
  Perturbative Methods for Nonlinear Differential Equations} and by 2012 FIRB \emph{Dispersive equations and Fourier analysis}, while M.\ Squassina
  is supported by 2009 PRIN \emph{Variational and Topological
 Methods in the Study of Nonlinear Phenomena}}
\address{Dipartimento di Matematica e Applicazioni
\newline\indent 
Universit\`a di Milano Bicocca
\newline\indent
Edificio U5, Via Roberto Cozzi 53, 20125 Milano, Italy}
\email{simone.secchi@unimib.it}

\author[M.\ Squassina]{Marco Squassina}
\address{Dipartimento di Informatica
\newline\indent
Universit\`a degli Studi di Verona
\newline\indent
C\'a Vignal 2, Strada Le Grazie 15, 37134 Verona, Italy}
\email{marco.squassina@univr.it}

\subjclass[2000]{35Q51,
  35Q40, 35Q41}
\keywords{Soliton dynamics, fractional Schr\"odinger equation, ground
  states}

\begin{document}

\begin{abstract}
  We investigate the soliton dynamics for the fractional
  nonlinear Schr\"odinger equation by a suitable modulational
  inequality. In the semiclassical limit, the solution concentrates 
  along a trajectory determined by a Newtonian equation depending of the fractional
  diffusion parameter.
\end{abstract}

\maketitle




\section{Introduction}

In the last years, the study of fractional integrodifferential equations
applied to physics as well as other areas has constantly grown. 
In \cite{metkla1,metkla2,hilfer}, the authors investigate recent developments in the description of anomalous diffusion
via fractional dynamics and many fractional partial differential
equations are derived asymptotically from L\'evy random walk models, extending
Brownian walk models in a natural way. In particular, in \cite{laskin} a fractional Schr\"odinger
equation was derived, extending to a L\'evy framework a classical result that path integral over Brownian trajectories leads to
the standard Schr\"odinger equation. We also refer the readers to \cite{rozmej} and to the references
therein for further bibliography on the subject.
Let $N\geq 1$, $s\in (0,1]$ and 
\[
0<p<\frac{2s}{N}.
\] 
Let $\mathrm{i}$ be the
imaginary unit and let $V$ denote a smooth external time-independent
potential.  The goal of this paper is the study of the behaviour of
the solution $u^\eps \colon \R^N\to\C$, $\eps>0$, to the Schr\"odinger equation
involving the fractional laplacian $(-\Delta)^s$
\begin{equation}
\label{eq:CP}
\left\{
\begin{array}{l}
 \mathrm{i} \eps \frac{\partial u^\varepsilon}{\partial t} = \frac{\eps^{2s}}{2} (-\Delta)^s u^\eps 
  + V(x) u^\eps - |u^\eps|^{2p}u^\eps    \quad\text{in $(0,\infty)\times\R^N$,}\\
  u^\eps(0,x)= Q\Big(\frac{x-x_0}{\eps} \Big) 
  e^{\frac{\mathrm{i}}{\eps} \langle x, v_0\rangle},
\end{array}
\right.
\end{equation}
in the semi-classical limit $\eps\to 0$, where $Q>0$ is the ground state
of
\begin{equation}
\label{ground-eq}
\frac{1}{2}(-\Delta)^s Q + Q =Q^{2p+1},    \quad\text{in $\R^N$,}
\end{equation}
and $x_0,v_0\in\R^N$, $v_0\neq 0,$ are the initial position and velocity
for the Newtonian type equation
\begin{equation}
\label{sysdin-s-intro}
s |\dot x|^{2s-2}\ddot x = -\nabla V(x),\qquad x(0) = x_0, \,\,\, \dot x(0) = v_0.
\end{equation}
In the limiting case $s=1,$ rigorous results about the soliton
dynamics of Schr\"odinger equation \eqref{eq:CP} were obtained in
various papers, among which we mention the contributions by Bronski
and Jerrard \cite{Bro-Jerr}, Keraani \cite{keraa} (see also
\cite{frolich1,BGM1,BGM2} where a different technique is used) 
via arguments based upon the
conservation laws satisfied by equation \eqref{eq:CP} and by the Newtonian ODE
\begin{equation}
\label{sysdin-s=1}
\ddot x=-\nabla V(x),
\qquad x(0) = x_0, \,\,\, \dot x(0) = v_0,
\end{equation}
combined with the modulational stability estimates due to Weinstein
\cite{weinstein1,weinstein2}.  Roughly speaking, the soliton dynamics
occurs when, choosing an initial datum behaving like $Q((x-x_0)/\eps)$
the corresponding solution $u^\eps(t)$ mantains the shape
$Q((x-x(t))/\eps)$, up to an estimable error and locally in time, in the semi-classical transition $\eps\to 0$. For a nice survey on solitons
and their stability features, see the work by Tao
\cite{tao-solitons}. Concerning the well-posedness of problem
\eqref{eq:CP} and a study of orbital stability of ground states, we
refer the reader to \cite{guo,guo2}.

To the best of our knowledge, in the fractional case $s\in (0,1)$
neither modulational inequalities nor a soliton dynamics behavior have
been investigated so far in the literature.  Recently there have been
many contributions concerning the properties of the solutions to
problem~\eqref{ground-eq}, with a particular emphasis on the their
qualitative behavior such as uniqueness, regularity, decays and --- more
important for our goals --- the nondegeneracy, namely the linearized
operator associated with \eqref{ground-eq} has an $N$-dimensional
kernel which is spanned by $\{\partial Q/ \partial x_j\}_{j=1,\ldots,N}$.

For these topics and the description of the physical
background, we refer the reader to the works by Lenzmann and Frank
\cite{Frank} in the one-dimensional case, and the work by Lenzmann,
Frank and Silvestre in the multi-dimentional setting \cite{FLS}. See
also the study of standing wave solutions in~\cite{cabre,felmer}, including
symmetry and regularity features.  \vskip3pt
\noindent 
Let $\mathcal{E}\colon H^s(\mathbb{R}^N,\mathbb{C})\to\mathbb{R}$ be
the energy functional defined by
\[
\mathcal{E}(u) := \frac{1}{2} \int |(-\Delta)^{\frac{s}{2}} u|^2 -
\frac{1}{p+1} \int |u|^{2p+2}
\]
and $\|\cdot\|_{H^s}$ denote the $H^1(\mathbb{R}^N,\mathbb{C})$-norm.
Then we have the following
\begin{theorem}
\label{thm:enconv}
Assume that
\[
0<s<1,\quad\,\, 0<p<\frac{2s}{N}.
\]
There exist positive constants $B,C$ independent of $\eps\in (0,1]$
and $s\in (0,1)$ such that
\[
\mathcal{E} (\phi) - \mathcal{E} (Q) \geq C
\inf_{x\in\mathbb{R}^N,\,\vartheta\in [0,2\pi)} \| \phi -
e^{\mathrm{i}\theta} Q(\cdot - x) \|^2_{H^s},
\]
for every $\phi\in H^s(\mathbb{R}^N,\mathbb{C})$ such that $\mathcal{E} (\phi) - \mathcal{E} (Q)\leq B$.
\end{theorem}

\noindent
This inequality is the fractional counterpart of an inequality which
follows as a corollary of the results by M. Weinstein on Lyapunov
stability for the nonlinear local Schr\"odinger equation, see
\cite{weinstein1,weinstein2}. A corresponding inequality 
for the nonlinear equations with a Hartree type nonlinearity was
obtained in \cite{PieMa} based upon the
nondegeneracy of ground states proved in \cite{Lenz}.
\vskip4pt
\noindent
Denoting $\|\cdot\|_{{\mathcal H}_\eps^s}^2=
\frac{1}{\eps^{N-2s}}\|(-\Delta)^{\frac{s}{2}}\cdot\|_2^2+\frac{1}{\eps^N}\|\cdot\|_2^2$,
we prove the following
\begin{theorem}
\label{primoteo}
Let $u^\eps(t)\in H^s(\R^N;\C)$ denote the unique solution to the
Cauchy problem~\eqref{eq:CP}.  Then there exists a positive constant
$C$, independent of $\eps\in (0,1]$ and $s\in (0,1)$, such that
\begin{equation}
\label{gradientvanish}
\|(-\Delta)^{\frac{s}{2}}u^\varepsilon(t)\|_2 \leq C \varepsilon^{\frac{N-2s}{2}},
\end{equation}
for every $t \geq 0$ and every $\varepsilon>0$.  Moreover, for any
$\eps>0$ sufficiently small and every $s\in (0,1)$ 
there exists a time $T^{\eps,s}>0$ and continuous functions
\[
\theta^{\eps,s}\colon [0,T^{\eps,s}]\to\R,\quad
z^{\eps,s}\colon\R^N\to\R,\quad \mathscr{E}\colon
[0,T^{\eps,s})\times(0,1]\times(0,1)\to\R,
\]
such that, uniformly on $s\in (0,1]$,
\[
\mathscr{E}(0,\eps,s)={\mathcal O}(\eps^2)
\]
and 
\[
\Big\|u^\eps(t)- e^{\frac{\mathrm{i}}{\eps}(\langle x,
  v(t)\rangle+\theta^{\eps,s}(t))}Q\Big(\frac{x-z^{\eps,s}(t)}{\eps}\Big)
\Big\|^2_{{\mathcal H}_\eps} \leq C\mathscr{E}(t,\eps,s)+{\mathcal
  O}(\eps^2)\quad\text{for all $t\in [0,T^{\eps,s})$.}
\]
Here $z^{\eps,s}(t)=x(t)+\eps\hat z^{\eps,s}(t)$ for some continuous  function $\hat
z^{\eps,s}\colon \R^N\to\R$, where $x(t)=x_s(t)$ is the solution to the
Cauchy problem~\eqref{sysdin-s-intro}.
\end{theorem}

\noindent
Hence, on a suitable time interval,  the solution remains close to the initial
profile with a term of order ${\mathcal O}(\eps^2)$. It is expected
that this qualitative behavior be preserved throughout the motion on
finite time intervals and also that $z^{\eps,s}(t)$ can be replaced by
$x(t)$ (solving problem~\eqref{sysdin-s-intro}) as in the local
case. On the other hand, the proof of this claim seems out of reach because
of the technical complications related to the nonlocal nature of
$(-\Delta)^s$ (see also Remark~\ref{mon-identities}).

\vskip2pt
\noindent
Furthermore, we have the following
\begin{theorem}
\label{thm:dyn}
Let $u^\eps_s(t)\in H^s(\R^N;\C)$ denote the unique solution to the
Cauchy problem~\eqref{eq:CP}.  Then it satisfies inequality
\eqref{gradientvanish}. Let $T>0$ and assume that
\[
0<s<1,\quad\,\, 0<p<\frac{2s}{N},
\]
that $V=V_1+V_2$ with $V_1\in {\mathcal C}^3(\mathbb{R}^N)$,
$V_2\in C^4(\mathbb{R}^N)$ and $V_2$ bounded from below. Then
there exist a positive constant $C$ and a continuous function
\[
\mathscr{A}\colon [0,T]\times (0,1]\times(0,1)\to\R,
\] 
such that 
\[
\lim_{s\to 1^-} \mathscr{A}(t,\eps,s)=0,\quad\text{for all $t\in
  [0,T]$ and $\eps\in (0,1]$}
\]
and
\[
\Big\|u^\eps_s(t)-Q\Big(\frac{x-x(t)}{\eps}\Big)
e^{\mathrm{i}\frac{\langle v(t), x\rangle}{\eps}}\Big\|_{{\mathcal H}_\eps^{s}}^2
\leq C\eps^{2s}+\|u^\eps_s(t)-u^\eps_1(t)\|_{{\mathcal
    H}_\eps^s}^2+{\mathscr A}(t,\eps,s), \quad\text{for all $t\in
  [0,T]$,}
\]
where $x(t)=x_s(t)$ is the solution 
to~\eqref{sysdin-s-intro}, provided $x_s$ converges to $x_1$ on $[0,T]$.
\end{theorem}

\noindent
Hence, on finite time intervals and precisely on the trajectory $x(t)$, the closeness estimate holds at the
weaker rate $\eps^{2s}$ and in terms of the distance
between the semigroups $u^\eps_s$ and $u^\eps_1$.

\begin{remark}
  A major difficulty in our analysis is the lack of a point-wise
  calculus for fractional derivatives. In particular, the fractional
  laplacian does not obey a point-wise chain rule, nor a point-wise
  Leibniz rule for products. Only approximate versions of the fractional
  chain rule hold: see for instance \cite[Lemma A10, Lemma A.11, Lemma
  A.12]{Killip} and the references therein.  
  This makes the analysis hard and we can prove the closedness of
  $u^\eps_s$ to the orbit $Q((x-x(t)/\eps)$ only when $s$ approaches the limit value $s=1$. 
%
We conjecture that the norm $\|u^\eps_s(t)-u^\eps_1(t)\|_{{\mathcal H}_\eps^s}$ vanishes in the limit $s \to 1$, but the proof seems out of reach so far, as a regularity theory for the solutions to the
fractional laplacian equation is still missing.
\end{remark}

\begin{remark}
If $x(t)$ solves \eqref{sysdin-s-intro}, then it is readily seen that
the energy $t\mapsto \frac{1}{2}|\dot x(t)|^{2s} + V(x(t))$
is a constant of motion. The Cauchy problems~\eqref{sysdin-s-intro} and \eqref{sysdin-s=1}
  are different from a dynamical viewpoint. For instance,   \eqref{sysdin-s-intro} could
  fail to have uniqueness of solutions in the case $s\in (1/2,1]$ since $|\xi|^{2-2s}\nabla V(x)$, where $\xi=\dot x$, could fail
  to be locally Lipschitz continuous.  Also, it could admit heteroclinic connections, 
  while \eqref{sysdin-s=1} does not, as easy examples in the case $N=1$ show.  
To compare the behaviour of systems~\eqref{sysdin-s-intro} and \eqref{sysdin-s=1}
in the physically relevant situation of harmonic potentials,
let $N=2$ and $V(x_1,x_2):=\frac{1}{2}x_1^2+2x_2^2$.
Then \eqref{sysdin-s-intro}, for $s\in (0,1]$ is
\begin{equation}
\label{sistema-harm}
\begin{cases}
\dot x_1=\xi_1, & \\
\dot x_2=\xi_2, & \\
\dot \xi_1=-\frac{1}{s}(\xi_1^2+\xi_2^2)^{1-s}x_1, & \\
\dot \xi_2=-\frac{4}{s}(\xi_1^2+\xi_2^2)^{1-s}x_2, & 
\end{cases}
\end{equation}
with initial datum $x_1(0)=1$, $x_2(0)=a$, $\xi_1(0)=1$ and $\xi_2(0)=b$ for some $a,b>0$. See Figures
\ref{casos1}-\ref{casos025} for the solutions to \eqref{sistema-harm} for the cases $s=1,1/2,1/4$
respectively and data $a=1,\,b=1/2$ (left) and $a=1/2,\,b=1$ (right). Clearly, the complexity of the solutions increases
as $s$ gets small. For any $s<1$, the system admits the stationary solutions of the form $(\alpha,\beta,0,0)$
for $\alpha,\beta\in\R$, while for $s=1$ it only admits the trivial stationary solution $(0,0,0,0).$
\begin{figure}[h!!!]
\begin{center}
   \begin{minipage}{0.49\textwidth}
      \includegraphics[scale=.64]{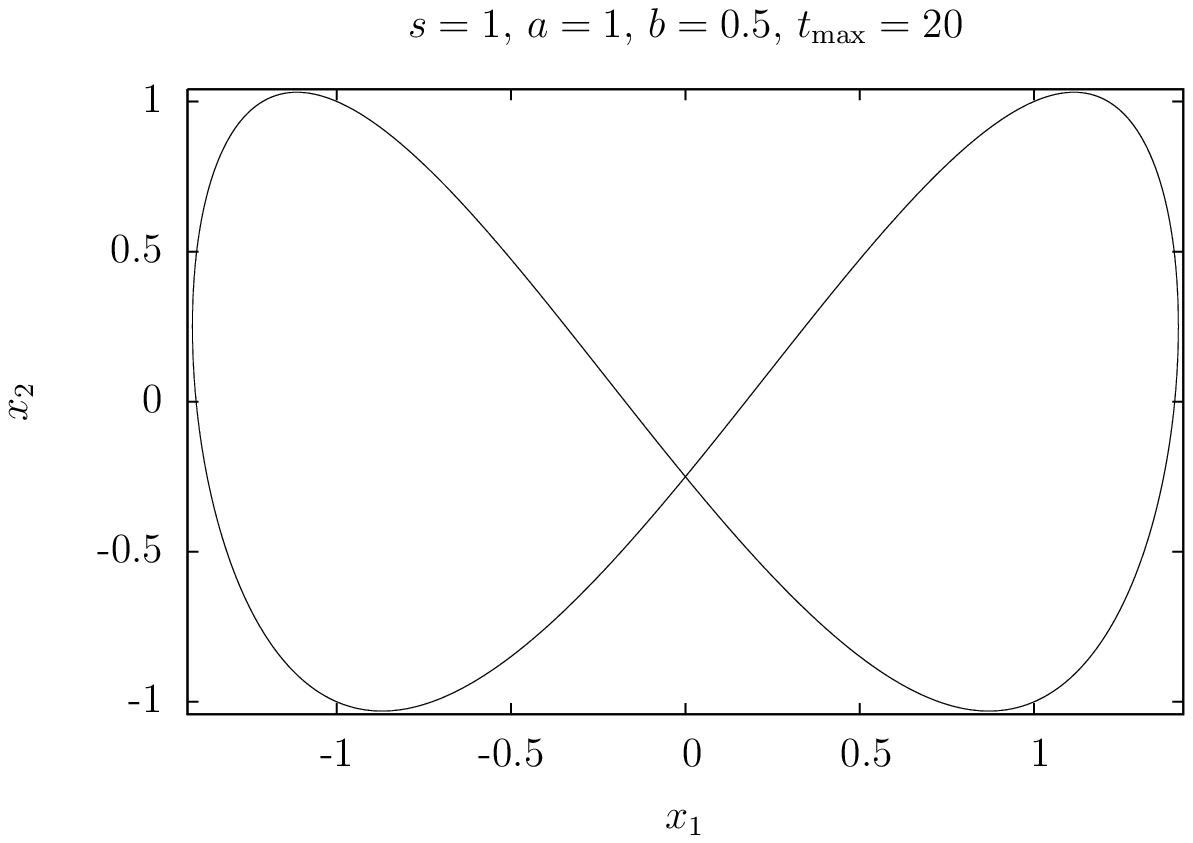}
   \end{minipage}
   \hspace{\fill}
   \begin{minipage}{0.49\textwidth}
      \includegraphics[scale=.64]{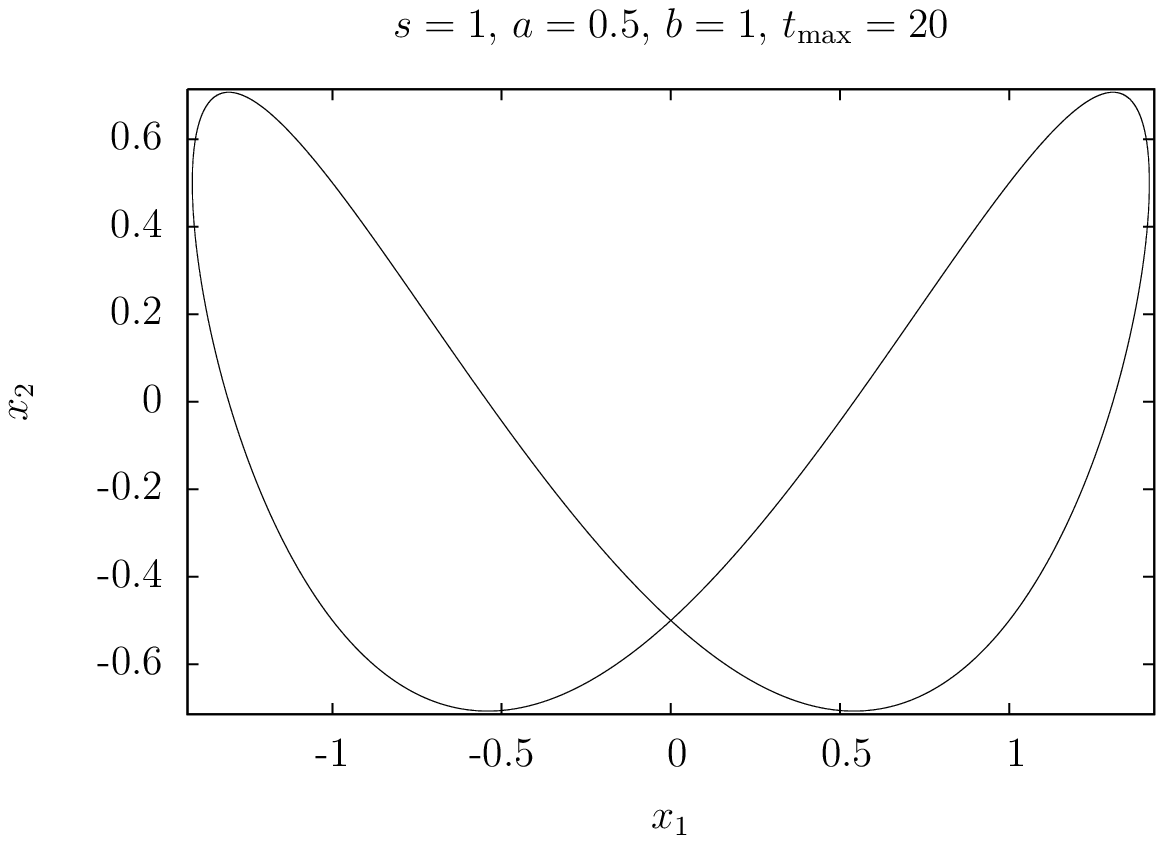}
   \end{minipage}
\end{center}
\caption{Solutions to~\eqref{sistema-harm} for $s=1$ with $a=1,\, b=0.5$ and
$a=0.5,\, b=1$.}\label{casos1}
\end{figure}

\begin{figure}[h!!!]
\begin{center}
   \begin{minipage}{0.49\textwidth}
      \includegraphics[scale=.64]{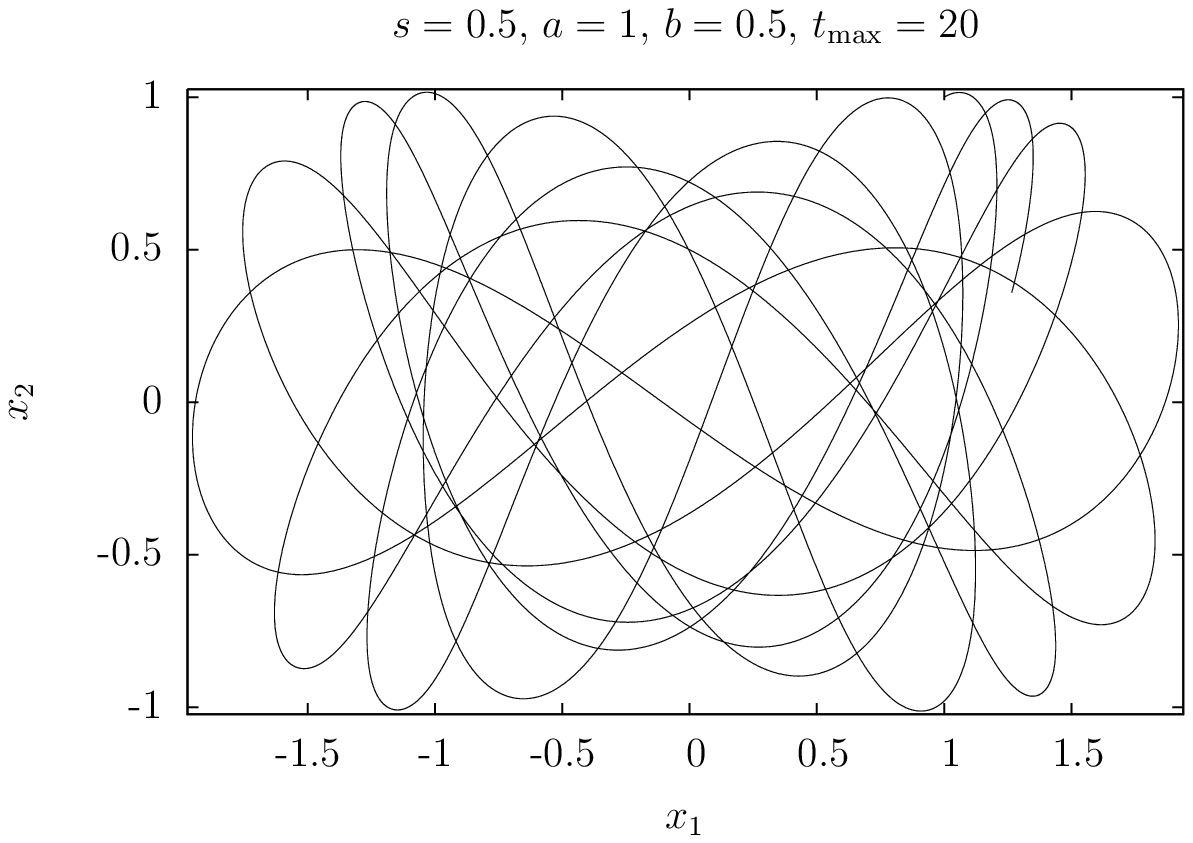}
   \end{minipage}
   \hspace{\fill}
   \begin{minipage}{0.49\textwidth}
      \includegraphics[scale=.64]{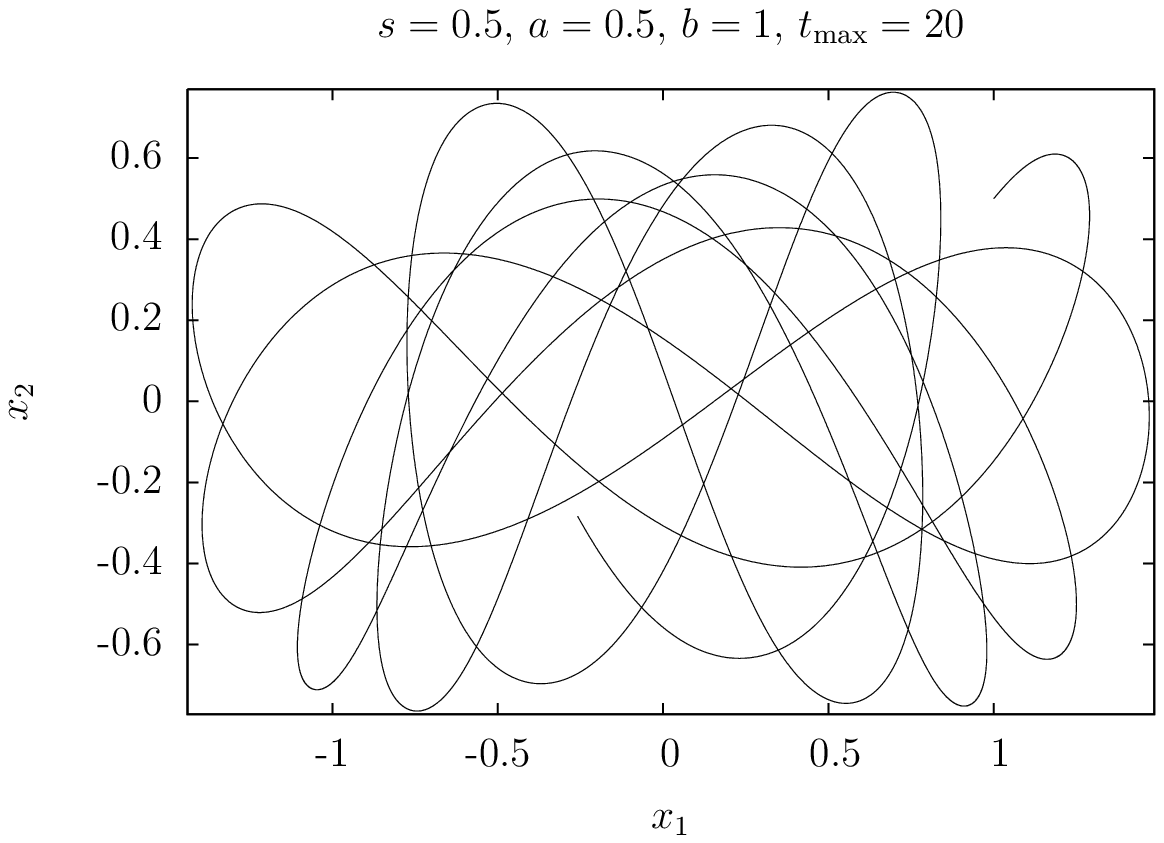}
   \end{minipage}
\end{center}
\caption{Solutions to~\eqref{sistema-harm} for $s=0.5$ with $a=1,\, b=0.5$ and
$a=0.5,\, b=1$.}\label{casos05}
\end{figure}

\begin{figure}[h!!!]
\begin{center}
   \begin{minipage}{0.49\textwidth}
      \includegraphics[scale=.64]{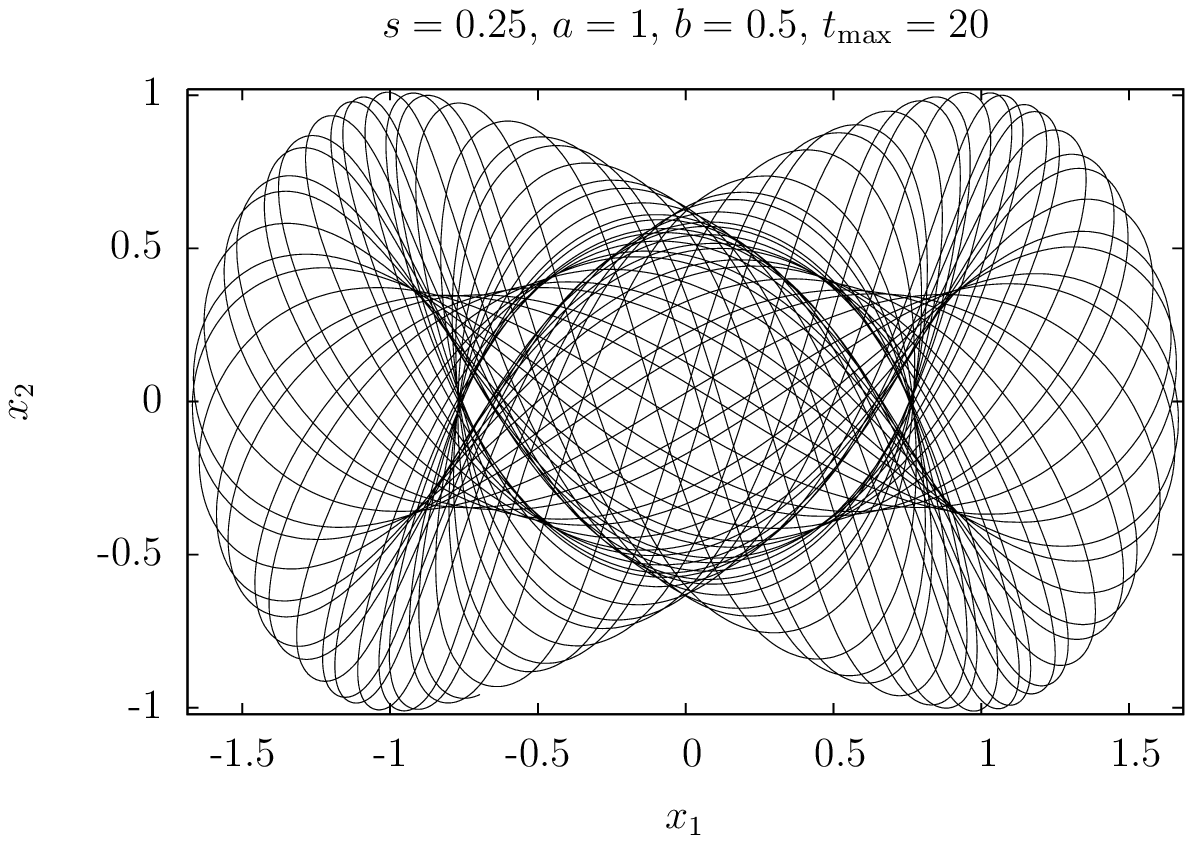}
   \end{minipage}
   \hspace{\fill}
   \begin{minipage}{0.49\textwidth}
      \includegraphics[scale=.64]{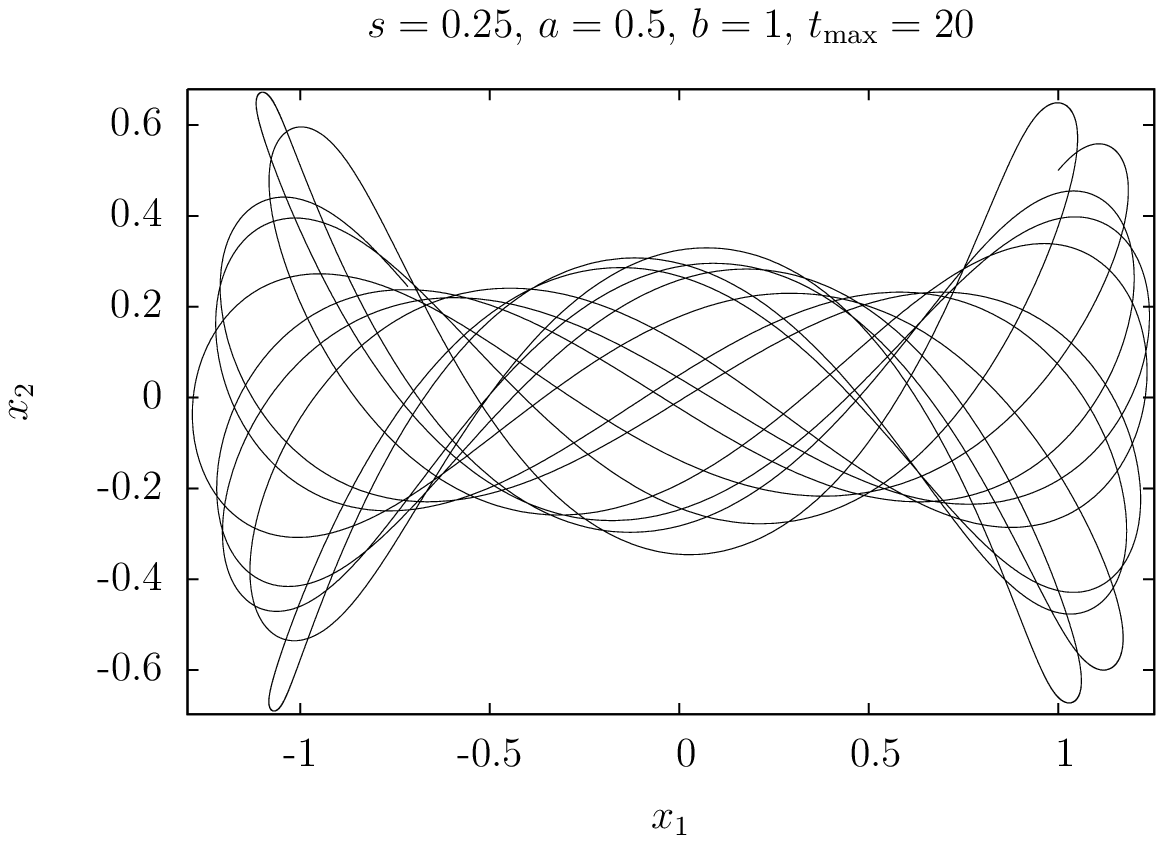}
   \end{minipage}
\end{center}
\caption{Solutions to~\eqref{sistema-harm} for $s=0.25$ with $a=1,\, b=0.5$ and
$a=0.5,\, b=1$.}\label{casos025}
\end{figure}
\end{remark}

\begin{remark}
A numerical analysis of the soliton dynamics behaviour according to Theorem~\ref{primoteo}
is currently under investigation and it will be the subject of a forthcoming manuscript.
\end{remark}

\subsection{Fractional laplacian and notations}
For the reader's convenience, we collect here some information about
the fractional laplacian $(-\Delta)^s$ in $\mathbb{R}^N$. We define 
it as the pseudo-differential operator acting on $u \in \mathscr{S}(\mathbb{R}^N,\mathbb{C})$ as
\[
(-\Delta)^s u := \mathcal{F}^{-1}\big( |\xi|^{2s} \mathcal{F}u(\xi)
\big),
\]
where $\mathcal{F}$ stands for the usual isometric Fourier transform
in $L^2(\mathbb{R}^N,\mathbb{C})$
\[
\mathcal{F}(u)(\xi) = \frac{1}{(2 \pi)^{N/2}} \int e^{-i
  \langle x,\xi \rangle} u(x)\, dx.
\]
As shown in \cite[Section 3]{DiNezza},
equivalent definitions are 
\begin{align*}
(-\Delta)^su(x) &= C(N,s) \, P.V. \int \frac{u(x)-u(y)}{|x-y|^{N+2s}}dy =
C(N,s)\lim_{\varepsilon \to 0} \int_{\mathbb{R}^N \setminus
  B(0,\varepsilon)} \frac{u(x)-u(y)}{|x-y|^{N+2s}}dy \\
  &= -\frac{1}{2} C(N,s) \int \frac{u(x+y)+u(x-y)-2u(x)}{|y|^{N+2s}}dy,
\end{align*}
where 
\begin{equation*} 
C(N,s) = \Big(\int \frac{1-\cos \zeta_1}{|\zeta|^{N+2s}}d\zeta\Big)^{-1}.
\end{equation*}
\begin{remark}
In some papers, the fractional laplacian is defined without any reference to the constant $C(N,s)$. This is legitimate when $s$ is kept fixed, but we will see that the behavior of $C(N,s)$ as $s \to 1$ will play a 
crucial r\^ole in Section \ref{preliminary}. 
\end{remark}

\noindent
The fractional Sobolev space $H^s(\mathbb{R}^N,\mathbb{C})$ may be
described as the set
\[
H^s(\mathbb{R}^N,\mathbb{C}) = \Big\{ u \in
  L^2(\mathbb{R}^N,\mathbb{C}) \mid \int \Big( 1+
  \tfrac{1}{2}|\xi|^{2s} \Big) |\mathcal{F}u(\xi)|^2 \, d\xi < +\infty
\Big\},
\]
endowed by the norm
\[
\|u\|_{H^s}^2 = \|u\|_2^2 + \frac{1}{2}\int |\xi|^{2s}  |\mathcal{F}u(\xi)|^2 \, d\xi=
\|u\|_2^2 + \frac{1}{2}\|(-\Delta)^{\frac{s}{2}}u\|_2^2.
\]
An identical (squared) norm is
\[
\|u\|_2^2 + \frac{C(N,s)}{4}\iint
\frac{|u(x)-u(y)|^2}{|x-y|^{N+2s}}dx dy,
\]
and, see \cite[Section 3]{DiNezza}, 
$$
\lim_{s\to 0^+}\frac{C(N,s) }{s(1-s)},\,\,
\lim_{s\to 1^-}\frac{C(N,s) }{s(1-s)}\in (0,+\infty).
$$
In the sequel, we will mainly work with the norm $\|u\|_2^2 + \frac{1}{2}
\|(-\Delta)^{\frac{s}{2}}u\|_2^2$.  
From the previous definitions, it follows that
\(
\big\| \sqrt{-\Delta}u \big\|_2 = \left\| \nabla u \right\|_2
\)
for any $u \in \mathscr{S}(\mathbb{R}^N)$.

\begin{remark} 
By equations (2.8) and (2.9) in
  \cite{DiNezza} and some elementary interpolation, we also deduce
  that the embeddings of $H^s(\mathbb{R}^N,\mathbb{C})$ have constants
  that can be considered as independent of $s \in [\delta,1]$, $\delta>0$. This fact
  will be used several times in the sequel.  Again from 
  \cite{DiNezza}, we have that $(-\Delta)^s u$ converges pointwise to
  $-\Delta u$ as $s\to 1^-$, for all $u\in
  C^\infty_c(\R^N)$. Furthermore, for $u \in H^1(\mathbb{R}^N,\mathbb{C})$,
\[
\lim_{s\to 1^-}\|(-\Delta)^{\frac{s}{2}}u\|_2=\|\nabla u\|_2.
\]
As a consequence, the fractional norms $\|u\|$ remain bounded as $s$
approaches $1$ and the Sobolev-Gagliardo-Nirenberg interpolation
inequality
\begin{equation} \label{eq:1.7}
\|u\|_{2p+2} \leq C \|u\|_2^{\alpha}
\|(-\Delta)^{\frac{s}{2}}u\|_2^{1-\alpha}, \quad \text{for all $u\in
  H^s(\mathbb{R}^N,\mathbb{C})$,}
\end{equation}
for a suitable $\alpha\in (0,1)$, holds with a contant $C$ which is
independent of the choice of $s\in (\delta,1].$
\end{remark}

\begin{center}
\textbf{Notation}
\end{center}
\begin{enumerate}
\item The usual euclidean scalar product of $\mathbb{R}^N$ will be denoted by \(\langle x,y \rangle = \sum_{j=1}^N x_j y_j\).
\item The space $\mathbb{C}$ will be endowed with the \emph{real}
  inner product defined by
  \begin{equation} \label{eq:innerC} 
  z \cdot w = \Re(z \overline{w})
    = \frac{z \overline{w}+\overline{z}w}{2}
\end{equation}
for every $z$, $w \in \mathbb{C}$.
\item We will denote by $\|\cdot \|_p$ the $L^p$-norm in
  $\mathbb{R}^N$, and by $\|\cdot \|_{H^s}$ the $H^s$-norm in
  $\mathbb{R}^N$. These norms come from the inner products
\begin{equation*} 
  \langle u,v \rangle_2 = \Re \int u \overline{v} \quad
  \text{and}\quad\langle u,v \rangle_{H^s} = \frac{1}{2}\Re \int
  (-\Delta)^{\frac{s}{2}}u \ \overline{(-\Delta)^{\frac{s}{2}} v}+\Re
  \int u \overline{v},
\end{equation*}
respectively.
\item Integrals over the whole space will be denoted by $\int$.
\item Generic constants will be denoted by the letter $C$. We shall always assume
  that $C$ may vary from line to line but it is {\em independent of} $s$ \emph{and} $\eps$ unless explicitly stated.
\item If $L$ is a linear operator acting on some space, the notation $\langle L,u \rangle$ denotes the value of $L$ evaluated at $u$. There is no confusion with the euclidean scalar product.
\end{enumerate}

\section{Properties of ground states}

\noindent A standing wave solution of the problem
\begin{equation*} 
\left\{
\begin{array}{ll}
  \mathrm{i} \frac{\partial \phi}{\partial t} -  \frac{1}{2}(-\Delta)^s \phi + |\phi|^{2p}\phi=0, \\[5pt]
  \phi(0,x)=\phi_0(x),
\end{array}
\right.
\end{equation*}
is a function of the form
\[
\phi(t,x)=e^{\mathrm{i}t}u(x),
\]
where $u \colon \mathbb{R}^N \to \mathbb{C}$ solves the elliptic equation
\begin{equation} \label{eq:2} 
\frac{1}{2}(-\Delta)^s u +u = |u|^{2p}u.
\end{equation}
\begin{definition}
  A solution $z \colon \mathbb{R}^N \to \mathbb{C}$ of (\ref{eq:2}) is
  called \emph{non-degenerate} if the set of solutions $u$ of the
  linearized equation
\[
\frac{1}{2}(-\Delta)^s u +u = (2p+1)|z|^{2p}u
\]
is the $N$-dimensional subspace spanned by the partial derivatives of
$z$.
\end{definition}
%
%
We recall the following facts from \cite{FLS,vald-1s}.
\begin{theorem} \label{th:2.3} Consider equation (\ref{eq:2}) for
  $0<s<1$ and $0<p<p_{\mathrm{max}}(s)$, where
\[
p_{\mathrm{max}}(s) = 
\begin{cases}
  \frac{2s}{N-2s} &\text{if $0<s < N/2$}\\
  +\infty &\text{otherwise}.
\end{cases}
\]
 Then the following facts hold.
\begin{itemize}
\item[(i)] \textbf{Existence.} There exists a solution $Q \in
  H^s(\mathbb{R}^N)$ of equation \eqref{eq:2} such that
  $Q$ is radially symmetric, positive and decreasing in
  $|x|$. Moreover, $Q$ is a ground state solution, namely a
  minimizer of the functional
\begin{equation*}
  J^{s,p}(u) = \frac{\left( \int |(-\Delta)^{s/2} u|^2 \right)^{\frac{p}{2s}} 
  \left( \int |u|^2 \right)^{\frac{p}{2s}(2s-1)+1}}{\int |u|^{2p+2}}.
\end{equation*}
\item[(ii)] \textbf{Symmetry and monotonicity.}  If $Q \in
  H^s(\mathbb{R}^N)$ solves (\ref{eq:2}) with $Q \geq 0$ and $Q$
  not identically equal to zero, then there exists $x_0 \in
  \mathbb{R}^N$ such that $Q(\cdot -x_0)$ is radially simmetric, positive and
  decreasing in $|x-x_0|$.
\item[(iii)] \textbf{Regularity and decay.} If $Q \in
  H^s(\mathbb{R}^N)$ solves (\ref{eq:2}), then $Q \in
  H^{2s+1}(\mathbb{R}^N)$. Moreover we have the decay estimate
\begin{equation*}
  |Q(x)| + |x \cdot \nabla Q(x)| \leq \frac{C}{1+|x|^{N+2s}}
\end{equation*}
for all $x \in \mathbb{R}$ and some constant $C>0$.
\item[(iv)] \textbf{Nondegeneracy.} Suppose $Q \in
  H^s(\mathbb{R}^N)$ is a solution of (\ref{eq:2}), and consider the
  linearized operator at $Q$
 \begin{equation*}
   L_{+} = \frac{1}{2}(-\Delta)^s + 1 - (2p+1)Q^{2p}
 \end{equation*}
 acting on $L^2(\mathbb{R}^N)$. If $Q=Q(|x|)>0$ is a ground state
 solution of (\ref{eq:2}), then
\begin{equation*}
  \ker L_{+} = \operatorname{span} \left\{ \frac{\partial Q}{\partial x_1},\ldots, \frac{\partial Q}{\partial x_N} \right\}.
\end{equation*}
\item[(v)] \textbf{Uniqueness.} The ground state for (\ref{eq:2}) is
  unique (up to translations).
\item[(vi)] \textbf{Stability.}  For every $s_0\in (0,1]$ and $Q=Q_s$, we have
\[
\sup_{s\in (s_0,1]}\|Q_s\|_{\infty}<\infty,\quad \sup_{s\in
  (s_0,1]}\|Q_s\|_{2}<\infty,\quad \sup_{s\in
  (s_0,1]}\|(-\Delta)^{s/2} Q_s\|_{H^s}<\infty.
\]
\end{itemize}
\end{theorem}
\begin{remark}
In the sequel, we will often write $Q$ instead of $Q_s$, when $s$ is
kept fixed.
\end{remark}

\noindent 
Let us introduce some notation. 
\begin{align*}
  I(u) &= \frac{1}{2}\mathcal{E}(u) + \frac{1}{2} \|u\|_2^2 \\
  \mathcal{M}_\gamma &= \left\{ u \in H^s(\mathbb{R}^N) \mid \|u\|_2^2=\gamma \right\} \\
  K_{\mathcal{E}} &= \left\{ c <0 \mid \text{$\mathcal{E}(u)=2c$, $\nabla_{\mathcal{M}_\gamma} \mathcal{E}(u)=0$ for some $u \in \mathcal{M}_\gamma$} \right\} \\
  \widetilde{K}_{\mathcal{E}} &= \left\{ u \in \mathcal{M}_\gamma \mid \nabla_{\mathcal{M}_\gamma} \mathcal{E}(u)=0 , \ \mathcal{E}(u)<0 \right\} \\
  K_I &= \left\{ m \in \mathbb{R} \mid \text{$I(u)=m$ and $I'(u)=0$ for some $u \in \mathcal{N}$} \right\} \\
  \widetilde{K}_I &= \left\{ u \in \mathcal{N} \mid I'(u)=0 \right\},
\end{align*}
where
\begin{equation*}
  \mathcal{N} = \left\{ u \in H^s (\mathbb{R}^N) \mid \langle I'(u),u \rangle =0 \right\}
\end{equation*}
is the Nehari manifold associated to (\ref{eq:2}).
For future reference, we record that, for any $\xi \in
H^s(\mathbb{R}^N,\mathbb{C})$ and any $\zeta \in
H^s(\mathbb{R}^N,\mathbb{C})$ there results
\begin{equation} \label{eq:3.1} 
\langle I''(\xi)\zeta,\zeta
  \rangle_{H^s} = \|\zeta\|_{H^s}^2 -2p \int \left( |\xi|^{2p-2} \left( \xi
      \cdot \zeta \right)\xi \right) \cdot \zeta - \int |\xi|^{2p}
  \zeta \cdot \zeta,
\end{equation}
where we have used the notation introduced in \eqref{eq:innerC}.
\begin{definition}
In the sequel, given a function $u$ and $\lambda,\mu\in\mathbb{R}$, 
we will write $u^{\mu,\lambda}(x)=\mu u (\lambda x)$.
\end{definition}
%
\begin{lemma} \label{lem:2.1} Given $u \in H^s(\mathbb{R}^N)$, the following scaling relations hold true:
\begin{align*}
  \|u^{\mu,\lambda}\|_2^2 &=\mu^2 \lambda^{-N} \|u\|_2^2, \\ 
  \|u^{\mu,\lambda}\|_{2p+2}^{2p+2} &= \mu^{2p+2} \lambda^{-N} \|u\|_{2p+2}^{2p+2}, \\
  \|(-\Delta)^{\frac{s}{2}} u^{\mu,\lambda}\|_2^2 &= \mu^2 \lambda^{2s-N} \|(-\Delta)^{\frac{s}{2}}u\|_2^2.
\end{align*}
\end{lemma}
\begin{proof}
  The three identities follow from a direct computation.
\end{proof}
\begin{lemma}
  Assume that
\[
0<s<1,\quad\,\, 0<p<\frac{2s}{N}.
\]
Then there is a bijective correspondence between the sets
$\widetilde{K}_\mathcal{E}$ and $\widetilde{K}_I$.
\end{lemma}
\begin{proof}
  Let us pick $v \in \mathcal{M}_\gamma$ such that $\langle
  \mathcal{E}'(v),v \rangle = -\ell \gamma$ and $\mathcal{E}(v)=2c<0$.
  Then 
$-\ell \gamma -4c = \langle \mathcal{E}'(v)v\rangle -2 \mathcal{E}(v) =-\frac{2p}{p+1}\|u\|_{2p+2}^{2p+2}<0,$
and therefore $\ell >0$. We can define a map $T^{\mu,\lambda} \colon \mathcal{M}_\gamma \to \mathcal{N}$ by 
$T^{\mu,\lambda} (v) = v^{\mu,\lambda}$,
where $\mu$ and $\lambda$ are defined by the condition
\[
\lambda = \ell^{-\frac{1}{2s}}, \quad \mu = \ell^{-\frac{1}{2p}}.
\]
It is easy to check that $v^{\mu,\lambda} \in \widetilde{K}_I$. Viceversa, if $u \in \widetilde{K}_I$, then we choose $\ell>0$ such that 
\begin{equation}
\label{normaelle2}
\ell^{\frac{1}{p}-\frac{N}{2s}} = \frac{\gamma}{\|u\|_2^2}, \quad \lambda = \ell^{\frac{1}{2s}}, \quad \mu = \ell^{\frac{1}{2p}},
\end{equation}
so that $u^{\mu,\lambda} \in \mathcal{M}_\gamma$ and $\nabla_{\mathcal{M}_\gamma} \mathcal{E}(u^{\mu,\lambda})=0$.
Whence
$\big( T^{\mu,\lambda} \big)^{-1} = T^{1/\mu,1/\lambda}$
concluding the proof.
\end{proof}

\begin{lemma}
\label{mappa}
Assume that
\[
0<s<1,\quad 0<p<\frac{2s}{N}.
\]
Then there exists a bijective correspondence ${\mathscr T} \colon K_I \to K_{\mathcal{E}}$ defined by the formula
\[
{\mathscr T}(m) = \left( \frac{N}{2s}-\frac{1}{p} \right) \left( \frac{\gamma sp}{2(p+1)s-Np}
\right)^{1+\frac{2sp}{2s-Np}} \left(\frac{1}{m} \right)^{\frac{2sp}{2s-Np}}.
\]
\end{lemma}
 
\begin{proof} 
Pick $m \in K_I$. Then, there is some $u \in \mathcal{N}$ such that $I(u)=m$ and $I'(u)=0$. Therefore
\[
m = I(u) -\frac{1}{2p+2} \langle I'(u),u \rangle  = \frac{1}{2} \left( 1 - \frac{1}{p+1} \right) \|u\|_{H^s}^2 >0.
\]
For $c \in K_{\mathcal{E}} \cap \mathbb{R}^{-}$ we select $v \in \mathcal{M}_\gamma$ 
corresponding to $c$. In turn, there exists $\ell>0$ such that
$\frac{1}{2}(-\Delta)^s v-|v|^{2p}v=-\ell v.$
Let us set $T^{\mu,\lambda}(v) = v^{\mu,\lambda}$ with
$\lambda = \ell^{-1/(2s)}$ and $\mu = \ell^{-1/(2p)}$.
Then, $T^{\mu,\lambda}$ maps ${\mathcal M}_\gamma$ into 
${\mathcal N}$ and $v^{\mu,\lambda}$ solves 
$\frac{1}{2}(-\Delta)^s v^{\mu,\lambda}+v^{\mu,\lambda}=|v^{\mu,\lambda}|^{2p}v^{\mu,\lambda}.$
The Poh\v ozaev identity yields 
\[
\frac{N-2s}{4} \int |(-\Delta)^{\frac{s}{2}}v^{\mu,\lambda}|^2  + \frac{N}{2} \|v^{\mu,\lambda}\|_2^2 = \frac{N}{2p+2} \|v^{\mu,\lambda}\|_{2p+2}^{2p+2}.
\]
But $v^{\mu,\lambda} \in \mathcal{N}$, namely
\[
\|v^{\mu,\lambda}\|_2^2 + \frac{1}{2}\int |(-\Delta)^{\frac{s}{2}}v^{\mu,\lambda}|^2 = \int |v^{\mu,\lambda}|^{2p+2}.
\]
Hence 
\[
\left( \frac{N-2s}{4}-\frac{N}{4p+4} \right) \|(-\Delta)^{\frac{s}{2}} 
v^{\mu,\lambda}\|_2^2 + \left( \frac{N}{2} - \frac{N}{2p+2} \right) \|v^{\mu,\lambda}\|_2^2=0,
\]
and
\[
\left( \frac{1}{4} - \frac{1}{4p+4} \right)  \|(-\Delta)^{\frac{s}{2}}v^{\mu,\lambda}\|_2^2 + \left( \frac{1}{2}-\frac{1}{2p+2} \right) \|v^{\mu,\lambda}\|_2^2  =m,
\]
where $m=I(v^{\mu,\lambda})$. After trivial manipulations, we discover that
\begin{align*}
\|(-\Delta)^{\frac{s}{2}}v^{\mu,\lambda}\|_2^2 &= \frac{2Nm}{s}, \\
\|v^{\mu,\lambda}\|_2^2 &= \frac{2ms(p+1)-Nmp}{sp}, \\
\|v^{\mu,\lambda}\|_{2p+2}^{2p+2} &= \frac{2m(p+1)}{p}.
\end{align*}
Recalling Lemma \ref{lem:2.1}, we write the previous identities as
\begin{align*}
\frac{\mu^2}{\lambda^{N-2s}} \|(-\Delta)^{\frac{s}{2}}v\|_2^2 &= \frac{mN}{s}, \\
\frac{\mu^{2p+2}}{\lambda^N} \frac{1}{2p+2} \int |v|^{2p+2} &= \frac{m}{p}, \\
\frac{\mu^2}{\lambda^N} \|v\|_2^2 &= \frac{2m(p+1)s - mNp}{sp}.
\end{align*}
But $v \in \mathcal{M}_\gamma$, and hence 
\begin{equation*}
\gamma = \|v\|_2^2 = \ell^{\frac{1}{p}-\frac{N}{2s}} \frac{2m(p+1)s - mNp}{sp},
\end{equation*}
and
\[
\ell^{\frac{2s-Np}{2sp}} = \frac{\gamma sp}{2m(p+1)s-mNp}.
\]
Since $\lambda = \ell^{-\frac{1}{2s}}$, $\mu = \ell^{-\frac{1}{2p}}$, we find 
\[
\|(-\Delta)^{\frac{s}{2}}v\|_2^2 = \frac{\lambda^{N-2s}}{\mu^2} \frac{2mN}{s} = \left( \frac{\gamma sp}{2(p+1)s - Np} 
\right) ^{1+\frac{2sp}{2s-Np}} \frac{2N}{s} m^{-\frac{2sp}{2s-Np}}.
\]
Similarly,
\[
\frac{1}{2p+2}\|v\|_{2p+2}^{2p+2} = \frac{\lambda^N}{\mu^{2p+2}} \frac{m}{p} = \frac{1}{p} \left( \frac{\gamma sp}{2(p+1)s-Np}
\right)^{1+\frac{2sp}{2s-Np}} \left(\frac{1}{m} \right)^{\frac{2sp}{2s-Np}}. 
\]
To summarize, if $c<0$ is a constrained critical value of $\mathcal{E}$ on $\mathcal{M}_\gamma$ and $m$ is the corresponding critical value of $I$, then $c$ is given by 
\[
c=\left( \frac{N}{2s}-\frac{1}{p} \right) \left( \frac{\gamma sp}{2(p+1)s-Np}
\right)^{1+\frac{2sp}{2s-Np}} \left(\frac{1}{m} \right)^{\frac{2sp}{2s-Np}}. 
\]
This concludes the proof.
\end{proof}

\noindent
We also have the following

\begin{corollary}
\label{GS-char}
Assume that
\begin{equation}
\label{assumpttcorollary}
0<s<1,\quad\,\, 0<p<\frac{2s}{N},\quad\,\,
\gamma_0:=m_{{\mathcal N}}\frac{2(p+1)s - Np}{sp},
\quad\,\, m_{{\mathcal N}}:=\inf_{u\in {\mathcal N}} I(u).
\end{equation}
Then we have
\[
m_{{\mathcal N}}=\inf_{u\in {\mathcal M}_{\gamma_0}}I(u)=:m_{\gamma_0}.
\]
Furthermore, any $u_0\in{\mathcal N}$ with $I(u_0)=m_{{\mathcal N}}$ satisfies $\|u_0\|_2^2=\gamma_0$ and
${\mathcal E}(u_0)=\inf_{u\in {\mathcal M}_{\gamma_0}}{\mathcal E}(u)$.
\end{corollary}
\begin{proof}
Observe that, taking into account the monotonocity of ${\mathscr T}$, we obtain 
\begin{align*}
m_{\gamma_0}&=\inf_{u\in {\mathcal M}_{\gamma_0}}\frac{1}{2}{\mathcal E}(u)+\frac{\gamma_0}{2}=
{\mathscr T}(m_{{\mathcal N}})+\frac{\gamma_0}{2}   \\
&=\Big( \frac{N}{2s}-\frac{1}{p} \Big) \Big( \frac{\gamma_0 sp}{2(p+1)s-Np}
\Big)^{1+\frac{2sp}{2s-Np}} \Big(\frac{1}{m_{{\mathcal N}}} \Big)^{\frac{2sp}{2s-Np}}+\frac{\gamma_0}{2} 
=m_{{\mathcal N}},
\end{align*}
after a few computations and by the value of $\gamma_0$. 
This concludes the proof of the first assertion. Now, given $u_0\in{\mathcal N}$ with $I(u_0)=m_{{\mathcal N}}$,
by repeating the argument in the proof of Lemma~\ref{mappa} (namely by 
combining the energy, the Pohozaev and the Nehari identities) and by the definition of $\gamma_0$ we get $\|u_0\|_2^2=\gamma_0$
(notice that, from~\eqref{normaelle2}, it holds $\ell=1=\lambda=\mu$, i.e.\ $T^{\mu,\lambda}=T^{1/\mu,1/\lambda}={\rm Id}$).
The last assertion then follows immediately from $m_{{\mathcal N}}=m_{\gamma_0}$.
\end{proof}

\begin{corollary}
\label{GS-char-II}
Let $Q>0$ be the unique ground state solution to problem \eqref{ground-eq} and
let $s,p$ and $\gamma_0$ be as in \eqref{assumpttcorollary}. Then we have
\begin{equation}
\label{minsphere}
{\mathcal E}(Q)=\min\{{\mathcal E}(q): q\in H^s(\mathbb{R}^N,\C),
\,\|q\|_2=\gamma_0=\|Q\|_2\},
\end{equation}
and $\min\{{\mathcal E}(q): q\in H^s(\mathbb{R}^N,\C),
\,\|q\|_2=\|Q\|_2\}$ admits a unique solution.
\end{corollary}
\begin{proof}
The assertion follows by Corollary~\ref{GS-char} and by the uniqueness of ground state solutions.
\end{proof}

\section{Spectral analysis of linearization}

\noindent
In this section we perform a spectral analysis of the linearized operator at a non degenerate ground state $Q$
\begin{equation*} 
L_{+} = \frac{1}{2}(-\Delta)^s +1 - (2p+1)Q^{2p}
\end{equation*}
acting on $L^2(\mathbb{R}^N,\mathbb{C})$. Let us introduce the closed subspaces of $H^s(\mathbb{R}^N,\mathbb{C})$
\begin{align*}
\mathcal{V} &= \left\{ u \in H^s(\mathbb{R}^N,\mathbb{C}) \mid \langle u,Q \rangle_2=0 \right\} \\
\mathcal{V}_0 &= \left\{ u \in H^s(\mathbb{R}^N,\mathbb{C}) \mid \left\langle u,Q \right\rangle_2=\Big\langle u,H(Q) \frac{\partial Q}{\partial x_j} \Big\rangle_2=0, \ j=1,2,\ldots,N \right\},
\end{align*}
where
$H(Q)=(2p+1)Q^{2p}$. 
\begin{lemma}
Assume that
\[
0<s<1,\quad 0<p<\frac{2s}{N}
\]
and define
\begin{equation*}
\alpha = \inf \left\{ \langle L_{+}(u),u \rangle \mid u \in \mathcal{V}_0,\ \|u\|_2=1 \right\}.
\end{equation*}
Then $\alpha > 0$.
\end{lemma}
\begin{proof}
Firstly, we claim that $\alpha \geq 0$. Indeed, $\partial Q/\partial x_j \in \mathcal{V}$ for each $j=1,\ldots,N$, and 
\[
\langle L_{+}(\partial Q/\partial x_j),\partial Q/\partial x_j\rangle =0.
\]
In addition, since (see Corollary~\ref{GS-char}) $Q$ minimizes $\mathcal{E}(u)$ over the constraint $\mathcal{M}=\{u \in H^s(\mathbb{R}^N,\mathbb{C}) \mid \|u\|_2 = \|Q\|_2 \}$, it follows that $Q$ also minimizes $2I(u)=\mathcal{E}(u)+\|u\|_2^2$ over the same constraint. In particular, $Q$ is a constrained critical point of $I$, and a direct computation shows that the second derivative $I''(Q)$ is positive semi-definite on $\mathcal{V}$. Therefore
\begin{equation} \label{eq:8}
\inf \left\{ \langle L_{+}(u),u \rangle \mid u \in \mathcal{V} \right\}=0.
\end{equation}
Since
\[
\alpha \geq \inf \left\{ \langle L_{+}(u),u \rangle \mid u \in \mathcal{V} \right\},
\]
the claim is proved.
We assume now, for the sake of contradiction, that $\alpha =0$. 
Pick any minimizing sequence $\{u_n\}_n$ for $\alpha$, so that $\|u_n\|_2=1$ for every $n \in \mathbb{N}$, $u _n \in \mathcal{V}_0$ and $\langle L_{+}(u_n),u_n \rangle=o(1)$ as $n\to\infty$. 
On the other hand,
\begin{equation*}
\langle L_{+}(u_n),u_n \rangle= 
\frac{1}{2}\int |(-\Delta)^{\frac{s}{2}} u_n|^2 + \int |u_n|^2  - (2p+1) \int Q^{2p} |u_n|^2,
\end{equation*}
and hence 
\begin{equation*}
\int_{\mathbb{R}^N} |(-\Delta)^{\frac{s}{2}} u_n|^2 \leq  C \left( o(1)+(2p+1) \int Q^{2p}|u_n|^2 \right)
\leq C+C \int |u_n|^2 \leq C.
\end{equation*}
The sequence $\{u_n\}_n$ being bounded in
$H^s(\mathbb{R}^N,\mathbb{C})$, we can assume without loss of
generality that $u_n \rightharpoonup u$ in
$H^s(\mathbb{R}^N,\mathbb{C})$, and $u \in \mathcal{V}_0$ because
$\mathcal{V}_0$ is weakly closed.

Notice that the operator $\{u \mapsto H(Q)u\}$ is a multiplication
operator by the function $Q^{2p}$ which tends to zero at
infinity. Given $\rho>0$, let us write
\[
\chi_\rho (x)=
\begin{cases}
1 &\text{if $|x| \leq \rho$}\\
0 &\text{if $|x|>\rho$}.
\end{cases}
\]
It follows that 
\begin{equation*}
  \int Q^{2p} |u|^2 - \left| \chi_\rho Q \right|^{2p} |u|^2 = \int_{\mathbb{R}^N \setminus B(0,\rho)} Q^{2p} |u|^2  \leq \sup_{x \in \mathbb{R}^N \setminus B(0,\rho)} Q(x)^{2p} \int |u|^2.
\end{equation*}
Then the compact embedding of $H^s(B(0,\rho))$ into $L^2(B(0,\rho))$
yields the compactness of the multiplication operator $H(Q)$ (see also
\cite[Theorem 10.20]{Weidmann}) and the convergence $\langle u_n ,
H(Q)u_n \rangle_2=\langle u,H(Q)u \rangle_2+o(1)$.  As a consequence,
\begin{equation*}
  0 \leq \langle L_{+}(u),u \rangle \leq \liminf_{n \to +\infty} \left( \|u_n\|_{H^s}^2 - \langle u_n, H(Q)u_n \rangle_2 \right) 
  = \lim_{n \to +\infty} \langle L_{+}(u_n),u_n \rangle =0,
\end{equation*}
forcing $\langle L_{+}(u),u \rangle =0$ and $\langle L_{+}(u_n),u_n
\rangle = \langle L_{+}(u),u \rangle+o(1)$.  By lower semicontinuity,
we get
\begin{align*}
  \|u\|_{H^s}^2 &\leq \liminf_{n \to +\infty} \|u_n\|_{H^s}^2 \leq \limsup_{n \to
    +\infty} \|u_n\|_{H^s}^2
  = \lim_{n \to +\infty} \langle L_{+}(u_n),u_n \rangle + \langle u_n , H(Q)u_n \rangle_2  \\
  &= \langle L_{+}(u),u \rangle + \langle u,H(Q)u \rangle_2 = \|u\|_{H^s}^2.
\end{align*}
So far we have proved that $u_n \to u$ strongly in
$H^s(\mathbb{R}^N,\mathbb{C})$ and that $u$ is a minimizer for
$\alpha$. From now on, for ease of notation, we assume that
$N=1$; the general case is similar, but we need to replace $Q'$ with either any partial derivative or with the gradient of $Q$ in the following arguments. Hence, the assumption reads as $p<2s$.  Let $\lambda$, $\mu$
and $\gamma$ be the Lagrange multipliers associated to $u$, so that,
for all $v \in H^s(\mathbb{R}^N,\mathbb{C})$,
\begin{equation*}
\langle L_{+}u,v \rangle =\lambda \langle u,v \rangle_2+\mu \langle Q,v \rangle_2 + \gamma \langle  H(Q)Q',v \rangle_2. 
\end{equation*}
Choosing $v=u \in \mathcal{V}_0$ immediately yields $\lambda=0$.
Instead, choosing $v = Q'$ and recalling also that $Q \perp Q'$ in $L^2(\mathbb{R}^N,\mathbb{C})$, we find
\begin{equation*}
0= \langle L_{+}u,Q' \rangle = \mu \langle Q,Q' \rangle_2 + \gamma \langle  H(Q)Q',Q' \rangle_2 = \gamma \langle H(Q)Q',Q' \rangle_2.
\end{equation*}
Now, 
\[
\langle H(Q)Q',Q' \rangle_2 = (2p+1) \int Q^{2p}|Q'|^2 >0,
\]
and this yields $\gamma =0$. Hence $L_{+}u=\mu Q$. To proceed further,
we compute 
\[
	L_{+}(x Q' ) = \frac{1}{2}(-\Delta )^s (x Q') +xQ' -(2p+1)Q^{2p}(x
Q')
\]
and we use the commutator identity (see \cite[Remark 2.2]{R-OS}
or \cite[Lemma 5.1]{Frank})
\begin{equation*}
 (-\Delta)^s(x \cdot \nabla u) =2s (-\Delta)^su + x \cdot \nabla (-\Delta)^s u
\end{equation*}
with $u=Q$, which implies 
\[
(-\Delta )^s (x Q') -x (-\Delta)^sQ' = 2s (-\Delta)^s Q.
\]
But $\frac{1}{2}(-\Delta)^s Q' + Q' -(2p+1)Q^{2p}Q'=0$
and hence 
\begin{equation} \label{eq:6}
L_{+}(xQ')=s (-\Delta)^s Q.
\end{equation}
Similarly,
\begin{equation} \label{eq:7}
L_{+} \Big( \frac{s}{p}Q \Big) = \frac{1}{2}(-\Delta)^s \frac{s}{p}Q + \frac{s}{p}Q - (2p+1) Q^{2p} \frac{s}{p}Q
= \frac{s}{p} \left(-2p Q^{2p} Q\right) = -2s Q^{2p+1}.
\end{equation}
Putting together \eqref{eq:6} and \eqref{eq:7} we see that
\[
L_{+} \Big( xQ' + \frac{s}{p}Q \Big) = -2sQ.
\]
As a consequence,
\[
L_{+}u = \mu Q = L_{+} \Big( -\frac{\mu}{2s} \Big( xQ' + \frac{s}{p}Q\Big) \Big).
\]
But $Q$ is a non degenerate ground state, namely 
$\ker L_{+} = \operatorname{span} \{Q'\}$, and there is $\vartheta \in \mathbb{R}$ with
\[
u + \frac{\mu}{2s} \Big( xQ' + \frac{s}{p}Q\Big) = \vartheta Q'.
\]
We claim that $\vartheta=0$. Indeed,
\[
u = -\frac{\mu}{2s} \Big( xQ' + \frac{s}{p}Q\Big) + \vartheta Q',
\]
and multiplying by $(2p+1)Q^{2p}$ we get
\[
(2p+1)Q^{2p}u = -\frac{\mu}{2s} (2p+1)Q^{2p} x Q' - \frac{\mu}{2p} (2p+1) Q^{2p} + (2p+1) \vartheta Q^{2p}Q'.
\]
Since $u \in \mathcal{V}_0$, 
\[
\langle (2p+1)Q^{2p}u , Q' \rangle_2 = \langle u, (2p+1) Q^{2p}Q' \rangle_2 =0.
\]
Since $Q$ is an even function, $Q'$ is an odd function, and this implies
\begin{align*}
\langle H(Q)Q,Q' \rangle_2 &= (2p+1) \int Q^{2p+1}Q' =0 \\
\langle H(Q)Q',Q' \rangle_2 &= (2p+1) \int Q^{2p}x (Q')^2=0.
\end{align*}
On the other hand,
\[
\langle H(Q)\vartheta Q',Q' \rangle_2 =(2p+1) \vartheta \int Q^{2p} (Q')^2 >0,
\]
and we conclude that $\vartheta =0$. hence
\[
u = -\frac{\mu}{2s} \Big( xQ'+\frac{s}{p}Q \Big)
\]
and
\[
0 = \int uQ = -\frac{\mu}{2s} \int x Q Q'  - \frac{\mu}{2p} \int Q^2.
\]
It is readily seen that $\mu\neq 0$. Moreover, an integration by parts shows that 
\[
\int x Q Q' = -\frac{1}{2} \int Q^2
\]
and thus 
\[
\Big( \frac{1}{2p}-\frac{1}{4s} \Big) \int Q^2=0.
\]
Since $p<2s$, we deduce $Q=0$, which is clearly impossible. The proof is complete.
\end{proof}
\begin{remark}
 Actually the previous proof yields a positive constant $\alpha_0$ such that
 \begin{equation*}
  \langle L_{+}(v),v \rangle \geq \alpha_0 \|v\|_2^2 \quad\text{for every $v \in \mathcal{V}_0$}.
 \end{equation*}
Hence $\mathcal{V}_0$ becomes a complete normed space with respect to the norm $v \mapsto \sqrt{\langle L_{+}v,v \rangle}$.
Now the Closed Graph Theorem tells us that, for a suitable $\overline{\alpha}>0$,
\begin{equation} \label{eq:19}
 \langle L_{+}(v),v \rangle \geq \overline{\alpha} \|v\|_{H^s}^2 \quad\text{for every $v \in \mathcal{V}_0$}.
\end{equation}
\end{remark}
\begin{lemma}
Suppose $\phi \in L^2(\mathbb{R}^N,\mathbb{C})$ satisfies $\|\phi\|_2=\|Q\|_2$. Then
\begin{equation} \label{eq:11}
 \langle Q,\Re (\phi-Q) \rangle_{2} = -\frac{1}{2} \left( \|\Re (\phi-Q)\|_2^2 + \|\Im (\phi-Q) \|_2^2   \right) 
 = -\frac{1}{2} \|\phi-Q\|_2^2.
\end{equation}
\end{lemma}
\begin{proof}
 It follows from a direct computation and the fact that $Q$ is real-valued. 
\end{proof}
\begin{proposition}
\label{stima+}
 Assume
 \[
  0<s<1, \quad 1<p<\frac{2s}{N}.
 \]
Let us take $\phi$ as in (\ref{eq:11}), such that
\begin{equation} \label{eq:12}
\Big\langle \Re (\phi-Q) , H(Q)\frac{\partial Q}{\partial x_j} \Big\rangle_2 =0 \quad\text{for $j=1,2,\ldots,N$}. 
\end{equation}
Then
\begin{equation} \label{eq:13}
\langle L_{+} (\Re (\phi-Q)),\Re (\phi-Q)\rangle \geq C \|\Re (\phi-Q)\|_{H^s}^2 - C_1 \|\phi-Q\|_{H^s}^4 - C_2 \|\phi-Q\|_{H^s}^3 
\end{equation}
for suitable constants $C$, $C_1$, $C_2>0$.
\end{proposition}
\begin{proof}
 It is not restrictive to fix $\|Q\|_2=1$. We decompose $U=\Re (\phi-Q)$ as $U=U_\parallel + U_\perp$, where $U_\parallel = \langle U,Q \rangle_2 Q$. By formula \eqref{eq:11}, we get
\begin{equation*}
\|(-\Delta)^{\frac{s}{2}}U\|_2^2 \leq 2 \|(-\Delta)^{\frac{s}{2}}U_\parallel\|_2^2 + 2 \|(-\Delta)^{\frac{s}{2}}U_\perp\|_2^2 = \frac{1}{2} \|\phi-U\|_2^4 \|(-\Delta)^{\frac{s}{2}}Q\|_2^2 + 2 \|(-\Delta)^{\frac{s}{2}}U_\perp\|_2^2,
\end{equation*} 
so that 
\begin{equation} \label{eq:15}
 \| (-\Delta)^{\frac{s}{2}} U_\perp\|_2^2 \geq \frac{1}{2}\| (-\Delta)^{\frac{s}{2}} U\|_2^2 - \frac{1}{4} \|\phi-Q\|_2^4 \| (-\Delta)^{\frac{s}{2}} Q\|_2^2.
\end{equation}
The symmetry of $L_{+}$ implies
 \begin{equation} \label{eq:14}
  \langle L_{+}U,U \rangle = \langle L_{+}U_\parallel,U_\parallel \rangle + 2 \langle L_{+}U_\perp , U_\parallel \rangle + \langle L_{+}U_\perp,U_\perp \rangle.
 \end{equation}
 But $\langle U_\parallel,H(Q)\partial Q / \partial x_j \rangle_2 =0$, hence also $\langle U_\perp , H(Q) \partial Q / \partial x_j \rangle_2 =0$ by (\ref{eq:12}).
As a consequence, $U_\perp \in \mathcal{V}_0$. We deduce from (\ref{eq:19}), (\ref{eq:11}) and (\ref{eq:15}) that
\begin{equation}\label{eq:16}
\langle L_{+}U_\perp , U_\perp \rangle \geq C \left( \|U\|_{H^s}^2 - \|\phi-Q\|_2^4 \right)
\end{equation}
Again, from \eqref{eq:11}, we get
\begin{align} \label{eq:17}
 \langle L_{+}U_\perp , U_\parallel \rangle &= \langle Q,U \rangle_2 \langle L_{+}U_\perp , Q \rangle = -\frac{1}{2} \|\phi-Q\|_2^2 \langle U_\perp,L_+Q \rangle \notag \\
&= \frac{p}{2} \|\phi-Q\|_2^2 \Big(\Re\int (-\Delta)^{s/2}U (-\Delta)^{s/2}Q-
\langle U,Q \rangle_2\|(-\Delta)^{s/2}Q\|^2_2\Big)  \notag\\
 &\geq - \frac{p}{2}\|\phi-Q\|_2^2 \|(-\Delta)^{s/2}(\phi-Q)\|_2\|(-\Delta)^{s/2}Q\|_2
 \geq -C\|\phi-Q\|_{H^s}^3.
\end{align}
Finally, we get
\begin{equation} \label{eq:18}
 \langle L_{+} U_\parallel,U_\parallel \rangle = \langle U,Q \rangle_2^2 \langle L_{+}Q,Q \rangle = \frac{1}{4} \|\phi-Q\|_2^4 \langle L_{+}Q,Q \rangle = -\frac{p}{2} \|Q\|_{H^s}^2 \|\phi-Q\|_2^4.
\end{equation}
Putting together (\ref{eq:14}), (\ref{eq:16}), (\ref{eq:17}) and (\ref{eq:18}), we complete the proof.
\end{proof}
\noindent
Let us denote by $L_{-}$ the imaginary part of the linearized operator at $Q$, namely
\begin{equation*} 
L_{-} = \frac{1}{2}(-\Delta)^s  + 1 - Q^{2p}.
\end{equation*}
\begin{proposition}
\label{stima-}
There results
\begin{equation*} 
\inf_{\substack{v \neq 0 \\ \langle v,Q \rangle_{H^s} =0}} \frac{\langle L_{-}v,v \rangle}{\|v\|_{H^s}^2} >0.
\end{equation*}
\end{proposition}
\begin{proof}
It suffices to prove that
\begin{equation} \label{eq:22}
\inf_{\substack{v \neq 0 \\ \langle v,Q \rangle_{H^s} =0}} \frac{\langle L_{-}v,v \rangle}{\|v\|_2^2} >0.
\end{equation}
First of all, let us recall that $\lim_{|x| \to +\infty}
Q(x)=0$. Since, as claimed in \cite[Section 3.2]{Frank},
\[
\sigma_{\mathrm{ess}}\left( \frac{1}{2}(-\Delta)^s+1 \right)=[1,+\infty)
\]
and since the multiplication operator by $Q^{2p}$ is compact, we deduce that
\[
\sigma_{\mathrm{ess}} \left( L_{-} \right)=[1,+\infty)
\]
It now follows that $L_{-}$ has a discrete spectrum over $(-\infty,1)$ which consists of eigenvalues of finite multiplicity. Of course $Q \in \ker L_{-}$, so that $0$ is an eigenvalue of $L_{-}$ and $Q$ is an associated eigenfunction. But $Q$ never changes sign, and we deduce from the proof of Lemma 8.2 in \cite{FLS} that $0$ is the smallest eigenvalue of $L_{-}$. In particular, $L_{-}$ is a non-negative operator. Once it is proved \cite{FLS} that the heat semigroup ${\mathcal H}_s(t)=\exp\{-t(-\Delta)^s\}$ is positivity preserving, namely its kernel is a positive function, standard arguments (see \cite[Section 10.5]{Teschl} or \cite[Theorems 10.32 and 10.33]{Weidmann}) show now that this eigenvalue is simple. Therefore, $\ker L_{-}=\operatorname{span}Q$. Let us set
\[
\omega = \inf \left\{\langle L_{-}v,v \rangle \mid \|v\|_2=1, \ \langle v,Q \rangle_{H^s} =0 \right\},
\]
and assume for the sake of contradiction that $\omega =0$. If $\{v_n\}_n$ is a minimizing sequence for $\omega$, it follows from the regularity properties of $Q$ that $\{v_n\}_n$ is bounded in $H^s(\mathbb{R},\mathbb{C})$, and we can assume without loss of generality that this sequence converges weakly to some $v$; as a consequence, $\langle v,Q \rangle_{H^s}=0$. Again, the compactness of the multiplication operator by $Q^{2p}$ entails
\[
0 \leq \langle L_{-}v,v\rangle \leq \liminf_{n \to +\infty} \Big( \|v_n\|_{H^s}^2 - \int Q^{2p}v_n^2 \Big) = \lim_{n \to +\infty} \langle L_{-}v_n,v_n \rangle =0,
\]
and thus $\langle L_{-}v,v \rangle =0$.
But then
\begin{align*}
\|v\|_{H^s}^2 \leq \liminf_{n \to +\infty} \|v_n\|_{H^s}^2 &\leq \limsup_{n \to +\infty} \|v_n\|_{H^s}^2 = \lim_{n \to +\infty} \Big( \langle L_{-}v_n,v_n \rangle + \int Q^{2p} v_n^2 \Big) \\
&= \langle L_{-}v,v \rangle + \int Q^{2p}v^2 \leq \|v\|_{H^s}^2.
\end{align*}
We have proved that $v_n \to v$ strongly, and that $v$ solves the minimization problem for $\omega$. Therefore, $\lambda$ and $\mu$ being two Lagrange multipliers, we have that
\begin{equation*}
\langle L_{-}v,\eta \rangle = \lambda \langle v,\eta \rangle_2 + \mu \langle Q,\eta \rangle_{H^s},
\end{equation*}
for every $\eta \in H^s(\mathbb{R},\mathbb{C})$.  Choosing $\eta=v$
yields $\lambda =0$; choosing $\eta=Q$ and recalling that $L_{-}Q=0$
yields $0=\langle v,L_{-}Q \rangle = \langle L_{-}v,Q \rangle = \mu
\|Q\|_{H^s}^2.$ Hence $\mu=0$, and we conclude that $L_{-}v=0$. Since we
know that $\ker L_{-}=\operatorname{span}Q$, for some $\theta \in
\mathbb{R}$ we must have $v=\theta Q$. But then $0=\theta \|Q\|_{H^s}^2$, a
contradiction. This shows that $\omega >0$, namely the validity of
(\ref{eq:22}).
\end{proof}
\begin{lemma}
Fix $\phi \in H^s(\mathbb{R}^N,\mathbb{C})$ such that $\|\phi\|_2 =
\|Q\|_2$ and
 \begin{equation} \label{eq:4.19}
  \inf_{\substack{x \in \mathbb{R}^N \\ \vartheta \in [0,2\pi)}} 
  \| \phi - e^{\mathrm{i}\vartheta} Q(\cdot - x) \|_{H^s} \leq \|Q\|_{H^s}.
 \end{equation}
Then
\begin{equation}\label{eq:3}
\inf\limits_{\substack{x \in \mathbb{R}^N \\ \vartheta \in [0,2\pi)}}
\| \phi - e^{\mathrm{i}\vartheta} Q(\cdot - x) \|_{H^s}^2,
\end{equation}
is achieved at some $x_0 \in \mathbb{R}^N$ and $\vartheta_0 \in
[0,2\pi)$. Moreover, writing $\phi(\cdot +x_0) e^{-\mathrm{i}\vartheta_0} = Q+W$ where $W=U+\mathrm{i}V$, we have
the relations, for $j=1,2,\ldots,N$:
\begin{equation} \label{eq:3.16}
  \left\langle U, H(Q)\frac{\partial Q}{\partial x_j} \right\rangle_2 =0\quad\text{and}\quad \langle V,Q \rangle_{H^s} =0.
\end{equation}
\end{lemma}
\begin{proof}
  The variable $\vartheta \in [0,2\pi)$ is clearly harmless, since
  $e^{\mathrm{i}\vartheta}$ describes the compact circle $S^1 \subset
  \mathbb{C}$. We can therefore assume that $\vartheta=0$. Consider
  the auxiliary function $\mathfrak{n} \colon \mathbb{R}^N \to
  \mathbb{R}$ defined by setting $\mathfrak{n}(x) = \left\| \phi -
    Q(\cdot - x) \right\|_{H^s}^2$.  Plainly, $\mathfrak{n}$ is a continuous
  function, and
  \begin{align*}
    \mathfrak{n}(x) &= 2 \|Q\|_2^2 + \| (-\Delta)^{\frac{s}{2}}Q\|_2^2
    +
    \| (-\Delta)^{\frac{s}{2}}\phi\|_2^2\\
    & - 2\Re \int_{{\mathbb R}^N} \overline{\phi(y)}Q(y-x)\, dy - \Re
    \int \overline{(-\Delta)^{\frac{s}{2}} \phi (y)}
    (-\Delta)^{\frac{s}{2}}Q(y)\, dy
  \end{align*}
  because $\|\phi\|_2 = \|Q\|_2$. Since both $Q(\cdot -x)$ and
  $(-\Delta)^{\frac{s}{2}}Q(\cdot -x)$ decay to zero as $|x| \to
  +\infty$ (thanks to Theorem \ref{th:2.3} and using the equation
  satisfied by $Q$), we deduce that they also converge weakly to zero
  as $|x| \to +\infty$. It easily follows that
  \[
  \lim_{|x| \to +\infty} \mathfrak{n}(x) = 2\|Q\|_2^2 +
  \|(-\Delta)^{\frac{s}{2}}Q\|_2^2 + \| (-\Delta)^{\frac{s}{2}} \phi
  \|_2^2 > \|Q\|_{H^s}^2.
  \]
  On the other hand, assumption (\ref{eq:4.19}) entails that, for
  every $\delta>0$, there exists a point $x_\delta \in \mathbb{R}$
  with $\mathfrak{n}(x_\delta) \leq \|Q\|_{H^s}^2 + \delta$. As a
  consequence, the function $\mathfrak{n}$ attains its infimum on some
  ball $B(0,R)$, for a suitable $R>0$, and the proof is complete.
  Finally, we compute the Euler-Lagrange equations associated to the
  variational problem \eqref{eq:3} by differentiating with respect to
  $\theta$ and to $x_j$:
\begin{align} 
&\left\langle \phi-e^{\mathrm{i}\vartheta_0}Q(\cdot-x_0),-\mathrm{i} e^{\mathrm{i}\vartheta_0}Q(\cdot-x_0) \right\rangle_{H^s} =0 \label{eq:5}\\
&\Big\langle \phi-e^{\mathrm{i}\vartheta_0}Q(\cdot-x_0) , -e^{\mathrm{i}\vartheta_0}\frac{\partial Q}{\partial x_j}(\cdot -x_0) \Big\rangle_{H^s} =0. \label{eq:6-bis}
\end{align}
%
Equation \eqref{eq:5} yields
\begin{multline*}
\Re \int \left( \phi - e^{\mathrm{i}\vartheta_0}Q(\cdot -x_0) \right)\ \overline{-\mathrm{i} e^{\mathrm{i}\vartheta_0}Q(\cdot -x_0)} \\
{}+ \frac{1}{2}\Re \int (-\Delta)^{\frac{s}{2}} \left(\phi - e^{\mathrm{i}\vartheta_0}Q(\cdot -x_0) \right) \overline{(-\Delta)^{\frac{s}{s}} \left( -\mathrm{i}e^{\mathrm{i}\vartheta_0}Q(\cdot -x_0) \right)}=0,
\end{multline*}
namely
\[
\int Q(\cdot -x_0) \Im \left(\phi e^{-\mathrm{i}\vartheta_0} \right) + \frac{1}{2}\int (-\Delta)^{\frac{s}{2}}Q(\cdot -x_0) \Im \left( (-\Delta)^{\frac{s}{2}} \left( e^{-\mathrm{i}\vartheta_0} \phi\right) \right)=0
\]
or $\langle Q,V\rangle_{H^s}=0$.
Similarly, equation \eqref{eq:6-bis} yields 
\begin{multline*}
\Re \int \left( \phi - e^{\mathrm{i}\vartheta_0}Q(\cdot -x_0) \right)\overline{-e^{\mathrm{i}\vartheta_0}\frac{\partial Q}{\partial x_j}(\cdot -x_0) } \\
{}+ \frac{1}{2} \Re \int (-\Delta)^{\frac{s}{2}} \left( \phi - e^{\mathrm{i}\vartheta_0}Q(\cdot -x_0) \right) \overline{(-\Delta)^{\frac{s}{2}} \left( -e^{\mathrm{i}\vartheta_0} \frac{\partial Q}{\partial x_j}(\cdot -x_0) \right)}=0,
\end{multline*}
or
\begin{equation*}
\int U \frac{\partial Q}{\partial x_j} + \int Q\frac{\partial Q}{\partial x_j} 
{}+ \frac{1}{2} \int (-\Delta)^{\frac{s}{2}} U \ (-\Delta)^{\frac{s}{2}} \frac{\partial Q}{\partial x_j} + \frac{1}{2}\int (-\Delta)^{\frac{s}{2}}Q \ (-\Delta)^{\frac{s}{2}} \frac{\partial Q}{\partial x_j} =0.
\end{equation*}
Since
\[
\int Q\frac{\partial Q}{\partial x_j} =0= \int  (-\Delta)^{\frac{s}{2}}Q \ (-\Delta)^{\frac{s}{2}} \frac{\partial Q}{\partial x_j},
\]
and using the fact that
\[
\frac{1}{2}(-\Delta)^{s} \frac{\partial Q}{\partial x_j} + \frac{\partial Q}{\partial x_j} = H(Q)\frac{\partial Q}{\partial x_j},
\]
we finally deduce
$\left\langle U , H(Q)\frac{\partial Q}{\partial x_j} \right\rangle_2=0.$
\end{proof}

\begin{lemma} \label{lem:3.7}
If $p\in (0,1)$, there exists a constant $C>0$ such that
\begin{equation*} 
\left| |z|^{p-1}z-|w|^{p-1}w \right| \leq C |z-p|^p,
\,\,\quad
\text{for every $z$, $w \in \mathbb{C}$.}
\end{equation*}
\end{lemma}
\begin{proof}
Let $z,w\in\mathbb{C}$ be given and 
let $\vartheta\in [0,2\pi)$ be the angle between them. Without
loss of generality, we may assume that $t=|z|/|w|>1$. Since we have
\begin{equation*}
\frac{\left| |z|^{p-1}z-|w|^{p-1}w \right|}{|z-w|^p} \leq
\sup_{ \substack{t \in [1,\infty)\\ \vartheta \in [0,2\pi)}}\frac{(t^{2p}+1-2 t^p\cos \vartheta)^{1/2}}{\left( t^2+1 - 2 t\cos \vartheta  \right)^{p/2}}<+\infty,
\end{equation*}
the assertion follows.
\end{proof}

\begin{proposition}
\label{regul}
Let $\Psi(u)=\int |u|^{2p+2}$. Then $\Psi$ is of class 
$C^2$ on $H^s(\R^N,\C)$ for $0<p<\frac{2s}{N}$.
\end{proposition}
\begin{proof}
Since $\Psi''$ is a symmetric bilinear form on the \emph{real} Hilbert space $H^s(\mathbb{R}^N,\mathbb{C})$, its norm as a bilinear form equals the norm of its associated quadratic form, see for example \cite[Lemma 2.1, pag. 173]{Edmunds}; therefore we prove that
\[
\lim_{v \to u} \sup_{h \neq 0} \frac{\Psi''(u)(h,h)-\Psi''(v)(h,h)}{\|h\|_{H^s}^2}=0.
\]
From (\ref{eq:3.1}) we know that $\Psi''(u)$ splits into two terms (we drop some multiplicative constants),
\[
\Psi_1''(u)(h,h):=\int |u|^{2p}h \overline{h} 
\quad\text{and}\quad \Psi_2''(u)(h,h):=\int |u|^{2p-2} (\Re (u \overline{h}))^2,
\,\,\quad h\in H^s(\R^N,\C),
\]
which we shall treat separately.
Let $\{u_n\}_n\subset H^s(\R^N,\C)$ be such that $u_n\to u$ strongly as $n\to\infty$. 
Then, in the case $2p\leq 1$,
by the H\"olderianity of the map $s\mapsto s^{2p}$ we obtain that
\begin{equation*}
|\Psi_1''(u_n)(h,h)-\Psi_1''(u)(h,h)|\leq C\int ||u_n|^{2p}-|u|^{2p}||h|^2
\leq C\int |u_n-u|^{2p}|h|^2.
\end{equation*}
By applying the H\"older inequality with admissible exponents $(q,r)$ respectively, 
\[
q:=\frac{N}{p(N-2s)}>1, 
\,\,\quad r:=\frac{N}{2ps+(1-p)N}\in \Big(1,\frac{N}{N-2s}\Big),
\]
it follows for every $h\in H^s(\R^N,\C)$ with $\|h\|_{H^s}\leq 1$
\begin{equation*}
|\Psi''_1(u_n)(h,h)-\Psi''_1(u)(h,h)|\leq C\|u_n-u\|_{\frac{2N}{N-2s}}^{2p}\|h\|_{2r}^2
\leq C\|u_n-u\|_{\frac{2N}{N-2s}}^{2p},
\end{equation*}
since $\|h\|_{2r}\leq C\|h\|_{H^s}\leq C$, concluding the proof for $\Psi_1''$.
The opposite case $2p>1$ can be treated similarly.
Let us now come to the treatment of $\Psi_2''$. We notice that, for $p<1$, we get
\begin{multline*}
\left| |u_n|^{2p-2} (\Re (u_n \overline{h}))^2 - |u|^{2p-2} (\Re (u \overline{h}))^2 \right| \\
\leq 2 |h| \max \left\{ |u_n|^{p},|u|^{p} \right\} \left| |u_n|^{p-1} \Re (u_n \overline{h}) - |u|^{p-1} \Re (u \overline{h}) 
\right| \\
\leq C \max \left\{ |u_n|^{p},|u|^{p} \right\} |u_n-u|^{p} |h|^2,
\end{multline*}
where we used Lemma \ref{lem:3.7}. 
Now we can proceed as before and conclude the proof.
\end{proof}

\subsection{Proof of Theorem \ref{thm:enconv}}

We consider the action $I(\phi) = \frac{1}{2}\mathscr{E}(\phi)+\frac{1}{2}\|\phi\|_2^2$
and we control the norm of $w$ in terms of the difference
$I(\phi) - I(Q)$.
Using the scale invariance of $I$, 
recalling that $\langle I'(Q),w \rangle = 0$,
the orthogonality conditions \eqref{eq:3.16},
Propositions \ref{stima+} and \ref{stima-}, and taking into account
Proposition~\ref{regul}, by virtue of Taylor formula, we have
\begin{align*}
I(\phi) - I(Q) &=  I(Q + w) - I(Q) = \langle I'(Q),w \rangle + \frac{1}{2}\langle I''(Q)w,w \rangle
+ o(\|w\|_{H^s}^2)  \\
&= \langle L_+ u,u \rangle + \langle L_- v, v \rangle + o(\|w\|_{H^s}^2) \\
\noalign{\vskip2pt}
&\geq C\|u\|_{H^s}^2+C\|v\|_{H^s}^2+ o(\|w\|_{H^s}^2)=C\|w\|_{H^s}^2+o(\|w\|_{H^s}^2).
\end{align*}
To complete the proof of Theorem \ref{thm:enconv}, we observe that
for every $\eps>0$ there exists $\delta>0$ such that, if $\phi\in H^s(\mathbb{R}^N,\mathbb{C}),$
$\|\phi\|_{2}=\|Q\|_2$ and ${\mathcal E}(\phi)-{\mathcal E}(Q)<\delta$, then
\[
\inf\limits_{x\in\mathbb{R}^N,\,\vartheta\in [0,2\pi)} \| \phi -
e^{\mathrm{i}\vartheta} Q(\cdot - x) \|_{H^s}<\eps.
\]
Then, choosing ${\mathcal E}(\phi)-{\mathcal E}(Q)$ small enough, Theorem \ref{thm:enconv} follows.
By the uniqueness of solutions to $\min\{{\mathcal E}(q): q\in H^s(\mathbb{R}^N,\mathbb{C}),
\,\|q\|_2=\|Q\|_2\}$ (see Corollary \ref{GS-char-II}) the above implication follows by 
Lions' concentration compactness principle as in \cite{CaLi}.
\qed

\section{Dynamics of the ground state}
\label{preliminary}

\noindent
We first recall the following (cf.\ \cite[Lemma 2.4]{vald-1s}).

\begin{lemma}\label{lemma_fall}
Let $s$, $\bar{\sigma} \in (0,1]$ and $\delta > 2 |\bar{\sigma}-s|$. Then, for any $\varphi \in H^{2(\bar{\sigma}+\delta)}(\mathbb{R}^N)$,
\[
\left\| (-\Delta)^{\bar{\sigma}}\varphi - (-\Delta)^s \varphi \right\|_2 \leq C(\bar{\sigma},\delta) |\bar{\sigma}-s| \left\| \varphi \right\|_{H^{2(\bar{\sigma}+\delta)}},
\]
for a suitable $C(\bar{\sigma},\delta)>0$ of the form $C(\bar{\sigma},\delta)=
\frac{C_1}{\bar\sigma}+\frac{C_2}{\delta}$ with $C_1,C_2$ independent of $\bar\sigma,\delta.$
\end{lemma}

\noindent
Let now $u^\eps$ be a solution of the Cauchy problem \eqref{eq:CP}. 
The energy is defined as
\[
E_\eps(t)=\frac{1}{2\eps^{N-2s}}\int |(-\Delta)^{\frac{s}{2}} u^\eps(t,x)|^2+\frac{1}{\eps^N}\int V(x)|u^\eps(t,x)|^2-
\frac{1}{(p+1)\eps^N}\int |u^\varepsilon(t,x)|^{2p+2},
\]
and $E_\eps(t)=E_\eps(0)$ for every $t\geq 0$.
Moreover the mass conservation reads as
\[
\frac{1}{\eps^N} \int |u^\eps(t,x)|^2  
=\|Q\|_2^2=:m,
\qquad \text{$t\geq 0$,\,\, $\eps>0$.}
\] 
\noindent
Let us set
\begin{equation*}
\J_s:=-C(N,s)\iint \frac{Q\left( x \right)(Q\left( x \right)-Q\left( x-z \right))
(1-\cos \langle z , v_0 \rangle )}{|z|^{N+2s}} dx  dz,
\end{equation*}
and define
\[
\mathcal{H}(t) := \frac{1}{2}m |v(t)|^{2s} + m V(x(t)),\quad t\geq 0.
\]
Then we have the following
\begin{lemma}
\label{espansione-energia-mod}
For $t\in [0,\infty)$ and $\eps>0$ we have
\[
E_\eps(t)= {\mathcal E}(Q)+{\mathcal H}(t)+{\mathcal O}(\eps^2)+\frac{1}{2}\J_s.
\]
Moreover, $\J_s={\mathcal O}(1-s)$.
\end{lemma}
\begin{proof}
Assuming $x_0=0$ for simplicity, we observe that
\[
\frac{1}{\varepsilon^{N-2s}}\iint \frac{\big| Q\left( \frac{x}{\varepsilon} \right) e^{\frac{\mathrm{i}}{\varepsilon} \langle x , v_0 \rangle} - Q\left( \frac{y}{\varepsilon} \right) e^{\frac{\mathrm{i}}{\varepsilon} \langle y,v_0 \rangle} \big|^2}{|x-y|^{N+2s}} dx  dy 
= \iint \frac{\big| Q\left( x \right) e^{\mathrm{i} \langle x,v_0 \rangle} - Q\left( y \right) e^{\mathrm{i} \langle y,v_0 \rangle} \big|^2}{|x-y|^{N+2s}} 
dx  dy.
\]
Recalling the identity \cite[formula (3.12)]{DiNezza}
\begin{equation}
\label{gr8}
\int \frac{1-\cos \langle z , v_0 \rangle}{|z|^{N+2s}}dz = \frac{|v_0|^{2s}}{C(N,s)},
\end{equation}
we obtain, on account of \cite[Proposition 3.4]{DiNezza}, the following conclusion
\begin{align*}
& \iint \frac{\left| Q\left( x \right) e^{\mathrm{i} \langle x,v_0 \rangle} - Q\left( y \right) e^{\mathrm{i} \langle y,v_0 \rangle} 
\right|^2}{|x-y|^{N+2s}} dx  dy \\
&= \iint \frac{\left| Q\left( x \right) e^{\mathrm{i} \langle x,v_0 \rangle} -Q(x) e^{\mathrm{i} \langle y,v_0 \rangle} +Q(x) e^{\mathrm{i} \langle y,v_0 \rangle} - Q\left( y \right) e^{\mathrm{i} \langle y,v_0 \rangle} \right|^2}{|x-y|^{N+2s}} dx  dy \\
&=  \iint \frac{|Q(x)-Q(y)|^2}{|x-y|^{N+2s}}dx \, dy +  \iint \frac{|Q(x)|^2 \left| e^{\mathrm{i} \langle x,v_0 \rangle} - e^{\mathrm{i} \langle y,v_0 \rangle} \right|^2}{|x-y|^{N+2s}} dx dy  +\frac{2}{C(N,s)}\J_s \\
&=\frac{2}{C(N,s)} \| (-\Delta)^{\frac{s}{2}} Q \|_2^2 +2\iint \frac{|Q(x)|^2  \left( 1- \cos \langle x-y,v_0 \rangle \right)}{|x-y|^{N+2s}} dx dy+\frac{2}{C(N,s)}\J_s \\
&= \frac{2}{C(N,s)} \| (-\Delta)^{\frac{s}{2}} Q \|_2^2 + 2\iint \frac{|Q(x)|^2  \left( 1- \cos \langle z , v_0 \rangle \right)}{|z|^{N+2s}} dx dz+\frac{2}{C(N,s)}\J_s \\
&= \frac{2}{C(N,s)} \| (-\Delta)^{\frac{s}{2}} Q \|_2^2 + 
2\int  |Q(x)|^2 \frac{|v_0|^{2s}}{C(N,s)}+\frac{2}{C(N,s)}\J_s  \\
&= \frac{2 }{C(N,s)} \big( \| (-\Delta)^{\frac{s}{2}} Q \|_2^2 + |v_0|^{2s} \|Q\|_2^2+{\J}_s \big).
\end{align*}
%
Therefore,
\begin{equation} \label{eq:4.24} 
\| (-\Delta)^{\frac{s}{2}} \big(
  Q \left( \cdot\right) e^{\mathrm{i}\langle \cdot,v_0 \rangle} \big)
  \|_2^2 = \| (-\Delta)^{\frac{s}{2}} Q \|_2^2 + |v_0|^{2s}
  \|Q\|_2^2 +\J_s.
\end{equation}
We know from a direct elementary computation (since
$\|(-\Delta)^{1/2}\varphi\|_2=\|\nabla\varphi\|_2$) that
\begin{equation} 
\label{eq:4.24-bis}
\big\| (-\Delta)^{1/2} \big( Q \left(\cdot \right) e^{\mathrm{i}\langle \cdot,v_0 \rangle} 
\big) \big\|_2^2 = \| (-\Delta)^{1/2} Q \|_2^2 + |v_0|^{2} \|Q\|_2^2.
\end{equation}
From Lemma \ref{lemma_fall}, we learn that
\begin{align*}
  \| (-\Delta)^{\frac{s}{2}} \big( Q \left(\cdot\right)
  e^{\mathrm{i}\langle \cdot,v_0 \rangle} \big) \|_2^2&= \big\|
  (-\Delta)^{1/2} \big( Q \left( \cdot\right) e^{\mathrm{i} \langle \cdot,v_0
    \rangle} \big) \big\|_2^2+{\mathcal O}((1-s)^2),
\\
\| (-\Delta)^{\frac{s}{2}} Q \|_2^2 &=\| (-\Delta)^{1/2} Q \|_2^2+{\mathcal O}((1-s)^2),
\notag
\end{align*}
Taking into account that $|v_0|^{2s}-|v_0|^2={\mathcal O}(1-s),$
it follows by comparing \eqref{eq:4.24} and \eqref{eq:4.24-bis} that
$\J_s={\mathcal O}(1-s)$. Whence, by energy conservation, we conclude that
\begin{multline*}
E_\varepsilon(t) = E_\varepsilon(0) = \frac{1}{2} \| (-\Delta)^{\frac{s}{2}} Q \|_2^2 + \frac{1}{2} |v_0|^{2s} \|Q\|_2^2 
+ \int V(\varepsilon x) |Q(x)|^2  - \frac{1}{p+1} \int |Q|^{2p+2} +\frac{1}{2}\J_s \\
= \mathcal{E}(Q) + \frac{1}{2}m|v_0|^{2s} + mV(0)-m V(0) + \int V(\varepsilon x) |Q(x)|^2  +\frac{1}{2}\J_s\\
= \mathcal{E}(Q) + \mathcal{H}(t) + \int V(\varepsilon x) |Q(x)|^2 \, dx - m V(0)+\frac{1}{2}\J_s.
\end{multline*}
It is readily checked that $\mathcal{H}$ is conserved 
along the trajectory $x(t)$, in light of equation~\eqref{sysdin-s-intro}.
Since the Hessian $\nabla^2 V$ is bounded and, by the radial symmetry of $Q$, 
\[
\int \langle x,\nabla V(0)\rangle |Q(x)|^2 =0,
\]
we conclude that $\int V(\varepsilon x)|Q(x)|^2  - m V(0) ={\mathcal O}(\varepsilon^2)$.
This ends the proof.
\end{proof}

\begin{remark}
  Unlike the local case $s=1$, in the cases $s\in(0,1)$ we cannot expect a precise
  conclusion as $E_\eps(t) = {\mathcal E}(r)+{\mathcal
    H}(t)+O(\eps^2)$. Indeed, the fractional Laplacian does not obey a
  Leibniz rule for differentiating products.
\end{remark}

\noindent
For the fractional norms of $u^\eps$, we have the following

\begin{lemma}
There exists a constant $C>0$ such that 
\[
\|(-\Delta)^{\frac{s}{2}}u^\varepsilon(t)\|_2 \leq C \varepsilon^{\frac{N-2s}{2}},
\]
for every $t \geq 0$ and every $\varepsilon>0$.
\end{lemma}
\begin{proof}
Since $V$ is bounded from below and $E_\varepsilon(t)$ is 
uniformly bounded with respect to $t\geq 0$, $\varepsilon>0$ and $s\in (0,1]$ 
by Lemma~\ref{espansione-energia-mod}, 
we deduce that, for all $t \geq 0$, 
\begin{multline} \label{eq:4.25}
\|(-\Delta)^{\frac{s}{2}}u^\varepsilon(t)\|_2^2 \leq 
C \varepsilon^{N-2s} + C\varepsilon^{-2s} \int |u^\varepsilon(t)|^{2p+2} \\
\leq  C \Big( \varepsilon^{N-2s} +  \varepsilon^{-2s} \|u^\varepsilon(t)\|_2^{2p+2-\frac{Np}{s}} \|(-\Delta)^{\frac{s}{2}} u^\varepsilon(t)\|_2^{\frac{Np}{s}} \Big).
\end{multline}
Here we have used the Sobolev-Gagliardo-Nirenberg inequality~\eqref{eq:1.7}
with exponent
\[
\alpha:=\frac{2s(p+1)-Np}{2s(p+1)}\in (0,1).
\]
Recalling that \(\|u^\varepsilon(t)\|_2 = \sqrt{m} \varepsilon^{N/2}\) by the conservation of the mass, we can write \eqref{eq:4.25} as
\begin{equation} \label{eq:4.26}
\|(-\Delta)^{\frac{s}{2}}u^\varepsilon(t)\|_2^2 \leq C \Big( \varepsilon^{N-2s} + \varepsilon^{-2s}\varepsilon^{\frac{N}{2}\left( 2p+2-\frac{Np}{s} \right)}\|(-\Delta)^{\frac{s}{2}}u^\varepsilon(t)\|_2^{\frac{Np}{s}} \Big).
\end{equation}
Now, setting for simplicity $\mathscr{N} = \mathscr{N}(\varepsilon) =\|(-\Delta)^{\frac{s}{2}}u^\varepsilon(t)\|_2>0$, (\ref{eq:4.26}) becomes
\[
\mathscr{N}^2 \leq C \left( \varepsilon^{N-2s} +  \varepsilon^{-2s}\varepsilon^{\frac{N}{2}\left( 2p+2-\frac{Np}{s} \right)}\mathscr{N}^{\frac{Np}{s}} \right).
\]
We claim that \(\mathscr{N} \leq C \varepsilon^{\frac{N-2s}{2}}\). Indeed, we rescale $\mathscr{N} = \varepsilon^{\frac{N-2s}{2}}\mathscr{Z}$ and deduce that
\[
\mathscr{Z}^2 \leq C (1 +  \mathscr{Z}^{\frac{Np}{s}} ).
\]
Since $Np<2s$ by assumption, we are lead to
\(\mathscr{Z} \leq C\) and the proof is complete.
\end{proof}

\vspace{3pt}
\noindent
Define now	
\[
\Psi^\eps(t,x):=\exp\Big(-\frac{\mathrm{i}}{\varepsilon}\langle\eps x+x(t), v(t)\rangle\Big) u^\eps (\eps x+x(t)),
\quad x\in {\mathbb R}^N,\,\, t\geq 0,
\]
where $(x(t),v(t))$ is the solution to problem~\eqref{sysdin-s-intro}. Notice that the exponential
function is a globally Lipschitz continuous complex valued function with modulus equal to one. Then, by a variant of
\cite[Lemma 5.3]{DiNezza}, it follows that $\Psi^\eps(t,\cdot)\in H^s({\mathbb R}^N,\C)$ for any $t\geq 0$
and $\eps>0$.  

\vskip3pt
\noindent
We have the following
\begin{lemma}
\label{energ2}
We have
\[
\mathcal{E}(\Psi^\eps(t))=\frac{1}{2}m|v(t)|^{2s}+\frac{\M(t,\varepsilon,s)}{2}
-\frac{1}{\eps^N}\int V(x)|u^\eps(t,x)|^2 +E_\eps(t),
\]
for every $t \geq 0$ and every $\varepsilon>0$.
\end{lemma}
\begin{proof}
Proceeding as in the proof of Lemma \ref{espansione-energia-mod}, we compute
\[
\int |(-\Delta)^{\frac{s}{2}} \Psi^\eps(t)|^2=\frac{C(N,s)}{2}\iint \frac{\left| \Psi^\varepsilon (t,x) - 
\Psi^\varepsilon (t,y) \right|^2}{|x-y|^{N+2s}} dx  dy = 
\I_1(t,\varepsilon,s) + \I_2(t,\varepsilon,s) + \M(t,\varepsilon,s),
\]
where we have set
\begin{align*}
\I_1(t,\varepsilon,s) &:= \frac{C(N,s)}{2}\iint \frac{\left| u^\varepsilon (t,\eps x+x(t)) - u^\varepsilon (t,\eps y+x(t)) \right|^2}{|x-y|^{N+2s}} dx dy \\
\I_2(t,\varepsilon,s) &:= \frac{C(N,s)}{2}\iint \big| u^\varepsilon (t,\varepsilon x+x(t)) \big|^2 
\frac{\big| e^{\frac{\mathrm{i}}{\varepsilon} \langle \varepsilon x+x(t),v(t) \rangle} - e^{\frac{\mathrm{i}}{\varepsilon} \langle \varepsilon y+x(t),v(t) \rangle} \big|^2}{|x-y|^{N+2s}} dx dy
\end{align*}
and
\begin{multline*}
\M(t,\varepsilon,s) := \\
C(N,s)\iint \Re\Big[u^{\varepsilon}(t,\varepsilon x+x(t)) \overline{\big[u^\varepsilon (t,\varepsilon x+x(t))-u^\varepsilon (t,\varepsilon y+x(t))\big]} \frac{e^{-\mathrm{i}\langle x-y , v(t) \rangle}-1}{|x-y|^{N+2s}}\Big]dx dy.
\end{multline*}
By changing variables, and recalling again~\eqref{gr8}, it readily follows that
\begin{align*}
\I_1(t,\varepsilon,s) &=  \varepsilon^{2s-N} \| (-\Delta)^{\frac{s}{2}} u^\varepsilon(t) \|_2^2, \\
\I_2(t,\varepsilon,s) &= \varepsilon^{-N} |v(t)|^{2s} \left\| u^\varepsilon(t) \right\|_2^2=
m|v(t)|^{2s}  \\
\M(t,\varepsilon,s) &= C(N,s)\varepsilon^{2s-N}  \iint \Re\Big[ u^{\varepsilon}(t,x) \overline{\big[u^\varepsilon (t,x)-u^\varepsilon (t,y)\big]} \frac{e^{-\frac{\mathrm{i}}{\eps}\langle x-y,v(t)\rangle}-1}{|x-y|^{N+2s}}\Big]dx  dy.
\end{align*}
It follows that
\begin{align*}
\mathcal{E}(\Psi^\eps(t))&= \frac{1}{2} \int |(-\Delta)^{\frac{s}{2}} \Psi^\eps(t)|^2  - \frac{1}{p+1} \int |\Psi^\eps(t)|^{2p+2} \\
&=\frac{1}{2}\frac{1}{\varepsilon^{N-2s}} \left\| (-\Delta)^{\frac{s}{2}} u^\varepsilon(t) \right\|_2^2+
\frac{1}{2}m|v(t)|^{2s}-\frac{1}{(p+1)\eps^N}\|u^\varepsilon(t,x)\|_{2p+2}^{2p+2}+\frac{\M(t,\varepsilon,s)}{2} \\
&=\frac{1}{2}m|v(t)|^{2s}+\frac{\M(t,\varepsilon,s)}{2}
-\frac{1}{\eps^N}\int V(x)|u^\eps(t,x)|^2 +E_\eps(t),
\end{align*}
concluding the proof.
\end{proof}
\noindent
Finally, we have the following
\begin{corollary}
\label{stima-energg}
There holds
\[
\mathcal{E}(\Psi^\eps(t))-{\mathcal E}(Q)=
{\mathscr E}(t,\eps,s)+{\mathcal O}(\eps^2),
\]
where ${\mathscr E}(t,\eps,s)={\mathscr E}_1(t,\eps,s)+{\mathscr E}_2(t,\eps,s)$ and
\begin{align*}
{\mathscr E}_1(t,\eps,s) &:=m|v(t)|^{2s}+\frac{\M(t,\varepsilon,s)+\mathbb{J}_s}{2},  \\
{\mathscr E}_2(t,\eps,s) &:=mV(x(t))-\frac{1}{\eps^N}\int V(x)|u^\eps(t,x)|^2,
\end{align*}
for every $t \geq 0$ and every $\varepsilon>0$. Furthermore ${\mathscr E}(0,\eps,s)={\mathcal O}(\eps^2)$.
\end{corollary}
\begin{proof}
By combining Lemma~\ref{energ2} with Lemma~\ref{espansione-energia-mod}, we find
\begin{align*}
\mathcal{E}(\Psi^\varepsilon(t)) &= \frac{1}{2} m|v(t)|^{2s} + \frac{1}{2}\mathbb{M}(t,\varepsilon,s)  - \frac{1}{\varepsilon^N} \int V(x) |u^\varepsilon(t,x)|^2  +E_\varepsilon(t)\\
&= \frac{1}{2} m |v(t)|^{2s} + \frac{1}{2} \mathbb{M}(t,\varepsilon,s) - 
\frac{1}{\varepsilon^N} \int V(x) |u^\varepsilon(t,x)|^2  \\
& + \mathcal{E}(Q) + \frac{1}{2}m|v(t)|^{2s} + mV(x(t))+{\mathcal O}(\varepsilon^2) + \frac{1}{2}\mathbb{J}_s \\
&= m|v(t)|^{2s} + \frac{\mathbb{M}(t,\varepsilon,s) + \mathbb{J}_s}{2}+\mathcal{E}(Q)+ mV(x(t))-
\frac{1}{\varepsilon^{N}} \int V(x)|u^\varepsilon(t,x)|^2 + {\mathcal O} (\varepsilon^2) \\
&={\mathscr E}_1(t,\eps,s)+{\mathscr E}_2(t,\eps,s)+ {\mathcal O} (\varepsilon^2).
\end{align*}
Now, since we have $u^{\varepsilon}(0,\varepsilon x+x(0))=Q(x)e^{\frac{i}{\eps}\langle \eps x+x_0,v_0\rangle}$,
we obtain
\begin{align*}
{\mathscr E}_1(0,\eps,s)&=m|v_0|^{2s}+\frac{\M(0,\varepsilon,s)}{2} \\
&\qquad {}-\frac{C(N,s)}{2}\iint \frac{Q\left( x \right)(Q\left( x \right)-Q\left( x-z \right))
(1-\cos\langle z,v_0 \rangle)}{|z|^{N+2s}} dx \, dz\\
&=
m|v_0|^{2s}+
\frac{C(N,s)}{2}\Re \iint Q(x)\Big[Q(x)-Q(y)e^{i\langle x-y, v_0\rangle}\Big] 
\frac{e^{-i\langle x-y,v_0\rangle}-1}{|x-y|^{N+2s}}dx dy \\
& -\frac{C(N,s)}{2}\iint \frac{Q\left( x \right)(Q\left( x \right)-Q\left( x-z \right))
(1-\cos\langle z,v_0 \rangle)}{|z|^{N+2s}} dx \, dz\\
&=
m|v_0|^{2s}-C(N,s)\int Q^2(x) \int\frac{1-\cos\langle z,v_0\rangle}{|z|^{N+2s}}dx dz=0.
\end{align*}
That ${\mathscr E}_2(0,\eps,s)={\mathcal O}(\eps^2)$ is immediately seen.
\end{proof}

\begin{remark}
\label{mon-identities}
From Corollary~\ref{stima-energg}, it seems evident that the quantity
\begin{multline*}
\varepsilon^{2s-N}\iint \frac{  \Re\big[  u^{\varepsilon}(t,x) \overline{[(u^\varepsilon (t,x)-u^\varepsilon (t,x-z))]} (e^{-\frac{\mathrm{i}}{\eps}\langle z,v(t)\rangle}-1)\big]
}{|z|^{N+2s}} dx dz\\
{}- \iint\frac{Q\left( x \right)(Q\left( x \right)-Q\left( x-z \right))
(1-\cos\langle z,v_0 \rangle)}{|z|^{N+2s}}dx dz,
\end{multline*}
multiplied by $C(N,s)/2$, represents a {\em nonlocal
counterpart} of the total momentum in the local case, precisely (compare ${\mathscr E}_1$ and ${\mathscr E}_2$
with the right-hand side of \cite[formula 3.5]{keraa})
\[
-\big\langle\dot x(t),\int p^\eps_{\mathrm{local}} (t,x)\big\rangle,\qquad
p^\eps_{\mathrm{local}} (t,x):=\frac{1}{\eps^{N-1}} 
\mathfrak{Im} (\bar{u}^\eps(t,x) \nabla u^\eps(t,x)), \quad x\in{\mathbb R}^N,\ t\in [0,\infty).
\]
As known, $p^\eps_{\mathrm{local}}$
satisfies the following identities, for $t\geq 0$ and $x\in{\mathbb R}^N$,
\[
\frac{\partial}{\partial t} \frac{|u^\eps(t,x)|^2}{\eps^3}=-\operatorname{div}(p^\eps_{\mathrm{local}}(t,x)), 
\qquad
\frac{\partial}{\partial t} \int  p^\eps_{{\rm local}} (t,x)\, dx
= 
- \frac{1}{\eps^N}\int \nabla V(x) |u^\eps(t,x)|^2\, dx.
\]
In the fractional case, a counterpart of these
identities seems hard to obtain.
\end{remark}

\subsection{Proof of Theorem~\ref{primoteo}}
By Corollary~\ref{stima-energg} and by the characterization of the ground states as minima on the sphere of $L^2$, we have
$0\leq \mathcal{E}(\Psi^\eps(t))-{\mathcal E}(Q)={\mathscr E}(t,\eps,s)+{\mathcal O}(\eps^2)$, where 
${\mathscr E}$ satisfies ${\mathscr E}(0,\eps,s)={\mathcal O}(\eps^2)$. By Theorem~\ref{thm:enconv}
we know that there exist constants $B,C>0$ such that for $\phi\in H^1(\mathbb{R}^3,\mathbb{C})$ with
$\| \phi \|_2 = \| Q \|_2$, we have
\[
\mathcal{E} (\phi) - \mathcal{E} (Q) \geq C
\inf_{x\in\mathbb{R}^3,\,\theta\in [0,2\pi)}
\| \phi - e^{i\theta} Q(\cdot - x) \|_{H^s}^2
\]
provided that $\mathcal{E} (\phi) - \mathcal{E} (Q)\leq B$. Then, introducing
\[
T^{\eps,s}:=\sup\Big\{t\in [0,T_0]\mid \mathscr{E}(\tau,\eps,s)\leq B\,\, 
\hbox{ for all } \tau\in [0,t]\Big\}
\]
and, since ${\mathscr E}(0,\eps,s)={\mathcal O}(\eps^2)$, it follows that $T^{\eps,s}>0$ 
for any $\eps>0$ sufficiently small and every $s\in (0,1)$ there exist families
of continuous functions $\theta^{\eps,s}\colon \R\to[0,2\pi)$ and 
$z^{\eps,s}\colon \mathbb{R}^N\to\mathbb{R}$ which satisfy the assertion.
\qed

\subsection{Proof of Theorem~\ref{thm:dyn}}
For $s\in (0,1]$, consider the 
solution $u^\eps_s(t,\cdot) \in H^s(\mathbb{R}^N,\mathbb{C})$ to the Cauchy problem \eqref{eq:CP}
Then, taking \cite[Proposition 2.2 and Lemma 5.3]{DiNezza} into account, there exists a 
positive constant $C$ such that 
\begin{equation*}
\Big\|u^\eps_s(t)-Q_s\Big(\frac{x-x_s(t)}{\eps}\Big) 
e^{\mathrm{i}\frac{\langle v_s(t), x\rangle}{\eps}}\Big\|_{{\mathcal H}_\eps^s}^2\leq
C\sum_{i=1}^4 \A_i(t;\eps,s),
\end{equation*}
where we have set
\begin{align*}
& \A_1(t;\eps,s):=\|u^\eps_s(t)-u^\eps_1(t)\|_{{\mathcal H}_\eps^s}^2, \\
& \A_2(t;\eps,s):=\frac{1}{\eps^{2(1-s)}}\Big\|u^\eps_1(t)-Q_1\Big(\frac{x-x_1(t)}{\eps}\Big) 
e^{\mathrm{i}\frac{\langle v_1(t), x \rangle}{\eps}}\Big\|_{{\mathcal H}_\eps^1}^2, \\
& \A_3(t;\eps,s):=\Big\|Q_1\Big(\frac{x-x_s(t)}{\eps}\Big) 
e^{\mathrm{i}\frac{\langle v_s(t), x\rangle}{\eps}}-Q_1\Big(\frac{x-x_1(t)}{\eps}\Big) 
e^{\mathrm{i}\frac{\langle v_1(t), x\rangle}{\eps}}\Big\|_{{\mathcal H}_\eps^s}^2,  \\
& \A_4(t;\eps,s):=\Big\|Q_s\Big(\frac{x-x_s(t)}{\eps}\Big) 
-Q_1\Big(\frac{x-x_s(t)}{\eps}\Big)\Big\|_{{\mathcal H}_\eps^s}^2,
\end{align*}
over finite time intervals $[0,T]$, for $T>0$. Then, we have the following
\begin{proposition}
There results
\begin{itemize}
\item[(a)] 
$\mathbb{A}_2(t;\varepsilon,s) \leq C \varepsilon^{2s}$ 
for every $\eps\in (0,1]$, $s\in (0,1)$, $t\geq 0$
and some $C>0$;
\item[(b)] $\lim\limits_{s\to 1^-}\mathbb{A}_3(t;\varepsilon,s) = 0$  
for every $\eps\in (0,1]$ and $t\geq 0$;
\item[(c)] $\lim\limits_{s\to 1^-}\mathbb{A}_4(t;\varepsilon,s) =0$ 
for every $\eps\in (0,1]$ and $t\geq 0$.
\end{itemize}
\end{proposition}
\begin{proof}
The proof of (a) follows immediately from \cite[Theorem 1.1]{keraa}. 
The proof of (b) is a consequence of the fact that $x_s(t) \to x_1(t)$ and $v_s(t) \to v_1(t)$ when $s \to 1$, since
\begin{align*}
\A_3(t,\eps,s)&\leq
C\Big\|Q_1\Big(\frac{\cdot-x_s(t)}{\eps}\Big)-Q_1\Big(\frac{\cdot-x_1(t)}{\eps}\Big) \Big\|_{{\mathcal H}_\eps^s}^2 
+\Big\|Q_1\Big(\frac{\cdot-x_1(t)}{\eps}\Big) 
\big[e^{\mathrm{i}\frac{\langle v_s(t), x \rangle}{\eps}}- 
e^{\mathrm{i}\frac{\langle v_1(t), x \rangle}{\eps}}\big]\Big\|_{{\mathcal H}_\eps^s}^2 \\
& =C\big\|Q_1(\cdot)-Q_1\Big(\cdot+\frac{x_s(t)-x_1(t)}{\eps}\Big) \big\|_{H^s}^2 
+\big\|Q_1(\cdot) \Xi_s(\cdot,t)\big\|_{H^s}^2 \\
&\leq C\big\|Q_1(\cdot)-Q_1\Big(\cdot+\frac{x_s(t)-x_1(t)}{\eps}\Big) \big\|^2_{H^1} 
+C\big\|Q_1(\cdot) \Xi_s(\cdot,t)\big\|^2_{H^1},
\end{align*}
where we have set
\[
\Xi_s(x,t):=e^{\mathrm{i} \langle v_s(t), x+\eps^{-1} x_1(t)\rangle}- 
e^{\mathrm{i} \langle v_1(t), (x+\eps^{-1} x_1(t)\rangle},\quad t\geq 0,\,\, x\in\R^N.
\]
The first term goes to zero as $s\to 1^-$, for any $\eps\in (0,1]$ and $t\geq 0$ 
(see e.g.\ \cite[p.185]{keraa}). Since $|\Xi_s(x,t)|\leq 2$ and 
$|\nabla \Xi_s(x,t)|\leq \|v_s\|_{L^\infty(0,T)}+\|v_1\|_{L^\infty(0,T)}$,
the second term goes to zero by dominated convergence.
The proof of (c) is a direct application of \cite[Lemma 2.6]{vald-1s}, since
\begin{align*}
\A_4(t,\eps,s)= \Big\|Q_s\Big(\frac{x-x_s(t)}{\eps}\Big) 
-Q_1\Big(\frac{x-x_s(t)}{\eps}\Big)\Big\|_{{\mathcal H}_\eps^s}^2
= \|Q_s-Q_1\|_{H^s}^2,
\end{align*}
concluding the proof.
\end{proof}

\noindent
Based upon the previous conclusions, the proof of Theorem~\ref{thm:dyn} is complete.

\bigskip
\bigskip
%

\bigskip

\end{document}